\documentclass[a4paper]{amsart}

\usepackage{verbatim,amsmath,amssymb,enumitem,amsthm}
\usepackage{mathtools}
\usepackage{tikz}

\theoremstyle{plain}   %% This is the default, anyway
\begingroup % Confine the \theorembodyfont command
\newtheorem{theorem}{Theorem}[section]   % Numbered within each section

\newtheorem{corollary}[theorem]{Corollary}     % Numbered along with thm
\newtheorem{lemma}[theorem]{Lemma}         % Numbered along with thm
\newtheorem{proposition}[theorem]{Proposition}  % Numbered along with thm

\newtheorem{question}[theorem]{Question}

\endgroup

\theoremstyle{definition}
\newtheorem{definition}[theorem]{Definition}   % Numbered along with thm

\theoremstyle{remark}
\newtheorem{remark}[theorem]{Remark}        % Numbered along with thm
\newtheorem{example}[theorem]{Example}        % Numbered along with thm

\newtheorem{notation}[theorem]{Notation}

\newcommand{\vf}{\varphi}

\newcommand{\spa}{\operatorname{span}}
\newcommand{\cal}{\mathcal}

\newcommand{\dist}{{\rm dist}\,}
\newcommand{\rd}{{\mathbb R}^d}
\newcommand{\sph}{{{\mathbb S}^{d-1}}}
\newcommand{\R}{{\mathbb R}}
\newcommand{\N}{{\mathbb N}}
\newcommand{\INt}{{\rm int}\,}
\newcommand{\ep}{\varepsilon}
\newcommand{\nor}{{\rm nor}\,}

\newcommand{\diam}{{\rm diam}\,}

\newcommand{\Tan}{{\rm Tan}\,}

\newcommand{\reach}{{\rm reach}\,}

\newcommand{\Ha}{{\cal H}}

\newcommand{\bD}{{\mathbb D}}

\newcommand{\lip}{\operatorname{Lip}}
\newcommand{\bw}{\textstyle{\bigwedge\nolimits}}
\newcommand{\arc}{{\rm arc}}

\newcommand{\eps}{\varepsilon}

\def\en{\mathbb N}
\def\er{\mathbb R}

\newcommand{\graph}{\operatorname{graph}}

\newcommand{\conv}{\operatorname{conv}}

\newcommand{\epi}{\operatorname{epi}}
\newcommand{\sepi}{\operatorname{epi_S}}
\newcommand{\hyp}{\operatorname{hyp}}
\newcommand{\shyp}{\operatorname{hyp_S}}

\newcommand{\llc}{\;\halfsq\;}
\newcommand{\lrc}{\;\ihalfsq\;}

\def\halfsq{\hbox{\kern1pt\vrule height 7pt\vrule width6pt height 0.4pt depth0pt\kern1pt}}
\def\ihalfsq{\hbox{\kern1pt \vrule width6pt height 0.4pt depth0pt
		\vrule height 7pt \kern1pt}}
  
\begin{document}

\title{On the structure of WDC sets}
\author{Du\v san Pokorn\' y, Jan Rataj and Lud\v ek Zaj\'\i \v cek}

\begin{abstract}
WDC sets in ${\mathbb R}^d$ were recently defined as sublevel sets of DC functions (differences of convex functions) at weakly regular values. They form a natural and substantial generalization of sets with positive reach and still admit the definition of curvature measures. Using results on singularities of convex functions, we obtain regularity results on the boundaries of WDC sets. In particular, the boundary of a compact WDC set can be covered by finitely many DC surfaces. More generally, we prove that any compact WDC set $M$ of topological dimension $k\leq d$ can be decomposed into the union of two sets, one of them being a $k$-dimensional DC manifold open in $M$, and the other can be covered by finitely many DC surfaces of dimension $k-1$. We also characterize locally WDC sets among closed Lipschitz domains and among lower-dimensional Lipschitz manifolds. Finally, we find a full characterization of locally WDC sets in the plane.
\end{abstract}

\date{\today}
\thanks{The authors were supported by the grant GACR No. P201/15-08218S}
\maketitle

\section{Introduction}
Federer in his fundamental paper \cite{Fe59} unified the approaches of convex and differential geometry, introducing curvature measures for sets with positive reach and proving the kinematic formulas. Quite recently, curvature measures have been defined for a substantially larger class of so-called (locally) WDC sets \cite{PR13}, and the corresponding kinematic formulas have been proved \cite{FPR15}. The basic difference between the two named set classes is that, while sets with positive reach are closely related to semiconvex functions of several variables, WDC sets are related to DC functions (i.e., differences of two convex functions) instead.

Following \cite{Fu94}, we say that a locally Lipschitz function $f:\rd\to [0,\infty)$ is an {\it aura} for a set $A\subset\rd$ if $A=f^{-1}\{0\}$ and $0$ is a weakly regular value of $f$ (i.e., there exist no sequences $x_i\to x$ and $u_i\to 0$ such that $f(x_i) >0=f(x)$  and  $u_i\in\partial f(x_i)$ are subgradients in the Clarke sense). 

This notion is motivated by the fact that $A$ has locally positive reach if and only if $A$ has a semiconvex aura \cite{Ba82}.  By the definition, $A$ is WDC if and only if it has a DC aura.
 So each set with locally positive reach is a   WDC set.

Because of the theory built in \cite{PR13} and \cite{FPR15}, the following rough question
 naturally arises: {\em  What is the structure of a general WDC set?} Note that, in contrast with sets with positive reach which are defined by the geometrically illustrative ``unique footpoint'' property, there seems to be no purely geometric property characterizing WDC sets. Also, there is a number of results on the structure of sets of positive reach, see, e.g., \cite{Fe59}, \cite{Ly05}, or the recent article  \cite{RZ}.
In the present article we prove some results on WDC sets, which are analogous to these results
on sets of positive reach.
 
\paragraph{\em Boundaries of WDC sets} 
We show that the boundary of a compact WDC set in $\R^d$ can be covered by finitely many DC hypersurfaces (i.e., graphs of Lipschitz DC functions of $d-1$ variables, see Proposition~\ref{wdcpokr}). Also, we show that a closed Lipschitz domain is locally WDC if and only if it is a closed DC domain (Theorem~\ref{wdcdom}).

\paragraph{\em Lower dimensional WDC sets}
Let $M\subset\R^d$ be a $k$-dimensional Lipschitz submanifold ($1\leq k<d$). Federer \cite[Remark~4.20]{Fe59} claimed that if $M$ has positive reach then it is already a $C^{1,1}$ manifold. In analogy to this result, we prove that if $M$ is a locally WDC set instead, it must be a DC manifold. We do not know whether the same is true for topological manifolds as it is (see \cite[Proposition~1.4]{Ly05}) in the case of sets with positive reach.

\paragraph{\em Structure of WDC sets}
We conjecture (see also Question~\ref{Q:decomposition}) that any compact WDC set $M\subset\R^d$ can be partitioned (``stratified'') into finitely many sets $M=T_1\cup\dots\cup T_m$, where each $T_i$ is a $k_i$-dimensional DC manifold, $0\leq k_i\leq d$, $i=0,\dots, m$. We prove a related easier result, namely that $M^{[0]}\cup\dots\cup M^{[d]}$ is open and dense in $M$, where $M^{[k]}$ is the set of all points $a\in M$ that agree with a $k$-dimensional DC manifold on a neighbourhood of $a$ (Corollary~\ref{slstr1}). This follows from our main result saying that for any relatively open subset $A\subset M$ of topological dimension $k$, $A\setminus A^{[k]}$ can be covered by finitely many DC surfaces of dimension $k-1$ if $k\geq 1$ (Theorem~\ref{struc}). This theorem is an analogue of a result of Federer on sets with positive reach, see Remark~\ref{ZFV}.

\paragraph{\em Planar locally WDC sets} 
In dimension $d=2$, we were able to prove a full (local) characterization of locally WDC sets. Roughly speaking, a set in $\R^2$ is locally WDC if and only if its complement can be locally represented as a disjoint finite union of sectors bounded by DC curves (see Theorem~\ref{T:characterisationInPlaneNew}). This proves, in particular, our conjecture on the structure of locally WDC sets in the planar case (see Remark~\ref{R:remarkToDecomposition} (ii)).

Throughout the paper we use two main technical tools. The first are results on singularities of convex and DC functions (partly proved in \cite{RZ}) which refine results of \cite{zajcon}. They show that certain singular sets (sets of nondifferentiability) of convex and DC functions can be covered by {\it finitely many} DC surfaces. These results are contained in Section~\ref{Sec_sing}.
 
The second main technical tool, which is quite essential for finer results on WDC sets, is a deformation lemma for Lipschitz functions applicable to Lipschitz and proper auras (Lemma~\ref{DL}), proving the existence of a deformation retraction for sublevel sets.
Using this deformation lemma, we present in Section~\ref{Sec_GB} also  a proof of the Gauss-Bonnet formula for sublevel sets at weakly regular values of Lipschitz, proper Monge-Amp\`ere functions (in particular, for WDC sets) relating the zero order curvature to the Euler-Poincar\'e characteristic (Proposition~\ref{P_GB}).
This is not a new result, the Gauss-Bonnet formula was proved in a more general context by Fu \cite{Fu94} and used in \cite{PR13}. The proof in \cite{Fu94}, however, uses some arguments which are only outlined. Since the Gauss-Bonnet formula is a cornerstone of the whole theory, we decided to provide a detailed proof, using the same idea as that of Fu, together with Lemma~\ref{DL}.

We would like to thank Joseph Fu for many helpful conversations.

\section{Preliminaries}
\subsection{Basic definitions}
We will use the notation $A^c$ for the complement of a set $A$. 
In any vector space $V$, we use the symbol $0$ for the zero element, $\overline{xy}$ for the closed segment with endpoints $x,y$ and $\spa M$ for the linear span of a set $M$. 
By a subspace of $V$ we always mean a linear subspace, unless specified otherwise. The symbol $B(x,r)$ ($\overline B(x,r)$) denotes the open (closed) ball with centre $x$ and radius $r$ in a metric space.

Let $X$ be a (real) Banach space with norm $|\cdot|$. If $x\in X$ and $x\in X^*$, we set  $\langle x, x^* \rangle \coloneqq  x^*(x)$. $\Tan(A,a)$ denotes the tangent cone of $A\subset X$ at $a\in X$. If $X$ is a Hilbert space and $V$ a closed subspace of $X$, we denote by $\pi_V$  the orthogonal projection to $V$.
 
We shall work mostly in the Euclidean space $\R^d$ with the standard scalar product $u\cdot v$ and norm $|u|$, $u,v\in\R^d$. The unit sphere in $\R^d$ will be denoted by $S^{d-1}$ and
 by $G(d,k)$ we denote the set of all $k$-dimensional linear subspaces of $\R^d$.

A mapping is called $K$-Lipschitz if it is Lipschitz with a constant $K$, and $\lip f$ denotes the (minimal) Lipschitz constant of $f$.
A bijection $f$ is called bilipschitz if both $f$ and $f^{-1}$ are Lipschitz. 
For a function $f:\R^d\to\R$ and $c\in\R$, we use the short notation $\{f\leq c\}$ for the set $\{x\in\R^d:\, f(x)\leq c\}$ (and analogously with other types of inequalities).

If $H$ is a finite-dimensional Hilbert space, $U\subset H$ open, $f:U\to\R$ locally Lipschitz and $x\in U$, we denote by $\partial f(x)$ the {\it subgradient of $f$ at $x$ in the Clarke sense}, which can be defined as the closed convex hull of all limits $\lim_{i\to\infty}f'(x_i)$ such that $x_i\to x$ and $f'(x_i)$ exists for all $i\in\N$. Since we identify $H^*$ with $H$ in the standard way, we sometimes consider $\partial f(x)$ as a subset of $H$. We will repeatedly use the fact that the mapping $x\mapsto\partial f(x)$ is upper semicontinuous and, hence (see \cite[Theorem~2.1.5]{Cl90}), 
\begin{equation}\label{uzcl}
v\in\partial f(x)\ \ \text{ whenever}\ \  x_i\to x,\  v_i\in\partial f(x_i)\ \ \text{ and}\ \  v_i\to v.
\end{equation}
We also use that $|u|\leq\lip f$ whenever $u\in\partial f(x)$, $x\in U$.

\subsection{DC functions, mappings, surfaces and manifolds}\textbf{}

 \begin{definition}\label{dc}
 Let $X$, $Y$ be finite-dimensional Banach spaces, $C\subset X$  an open convex set and $G\subset X$ an open  set.
\begin{enumerate}
\item[(i)]
A real function on $C$ is called a DC function if it is a difference of two  convex functions.
\item[(ii)]
We say that a mapping $F:C \to Y$ is DC if $y^* \circ F$ is DC for every functional $y^* \in Y^*$.
\item[(iii)]
We say that $f:G \to \R$ (resp. $F:G \to Y$) is locally DC if for each $x \in G$ there exists
 $\delta>0$ such that $f$ (resp. $F$) is DC on $B(x, \delta)$.
\end{enumerate}
\end{definition}
We will use the following well-known properties of DC functions and mappings.

\begin{lemma}\label{vldc}
Let $X$, $Y$, $C$, $G$ be as in Definition \ref{dc}. Then the following assertions hold.
\begin{enumerate}
\item[{\rm (i)}]
	If $f: C\to Y$ and $g: C\to Y$ are DC, then (for each $a\in \R$, $b\in \R$) the mappings $af + bg$ is DC.
	 If $Y= \R$, then also
	 $\max(f,g)$ and $\min(f,g)$ are DC.
\item[{\rm (ii)}]
$F:C \to Y$ is DC if and only if $y^* \circ F$ is DC for each
  $y^*$ from a basis of $Y^*$.
\item[{\rm (iii)}]	
 If  $F: G \to Y$ is $C^{1,1}$ (i.e., $F$ is differentiable and
 the derivative $x \mapsto F'(x)$ is Lipschitz on $G$), then $F$ is locally DC. In particular
 (cf. (iv)) each affine $F:X \to Y$ is DC.
	\item[{\rm (iv)}]
	 Each locally DC  mapping  $F:C \to Y$ is DC.
	\item[{\rm (v)}]
	Each locally DC  mapping  $F:G \to Y$  is locally
  Lipschitz. 
	\item[{\rm (vi)}]
	If $Z$ is a finite-dimensional Banach space,  $H\subset Y$ is  open, $f: G \to Y$ is locally DC,  $g:  H\to Z$ is locally DC and $f(G) \subset H$, then $g \circ f$ is locally DC on $G$.
	\item[{\rm (vii)}]
	Let  $\dim X = \dim Y$, $H\subset Y$ be  open, and let $f:G \to H$ be a bilipschitz
	 locally DC mapping. Then $f^{-1}$ is locally DC on $H$.     
	\item[{\rm (viii)}] 
	If $f: C \to \R$ is  DC, $D \subset X$ is open convex such that
	$\overline D\subset C$ and  $\overline D$
	 is compact, then there exists a Lipschitz DC function $\tilde f: X \to \R$ such that
	 $\tilde f(x) = f(x),\ x \in \overline D$.
	\item[{\rm (ix)}]
Let $F_i: C \to Y$, $i=1,\dots,m$, be DC mappings. Let $F: C \to Y$ be a continuous mapping
 such that $F(x) \in \{F_1(x),\dots,F_m(x)\}$ for each $x \in C$. Then $F$ is DC on $C$.
\end{enumerate}
\end{lemma}
\begin{proof}
Property (i) follows easily from definitions, see e.g. \cite[p. 84]{Tuy}, and (ii) follows from (i). For (iii), see e.g. \cite[Proposition 1.11]{VZ}.
Property (iv) was proved in \cite{Har}.
 Property (v) follows from the local Lipschitzness of  convex functions. 
Statement (vi) is ``Hartman's superposition theorem'' from \cite{Har}; for the proof see also \cite{Tuy} or 
  \cite[Theorem 4.2.]{VZ}. Assertion (vii) follows from   \cite[Theorem 5.2.]{VZ}.
	To show (viii), write $f=g-h$, where $g$, $h$ are convex on $C$, and observe that $g$ and $h$
 are bounded on some open convex $E$ with $\overline D\subset E\subset C$. Therefore
 (see e.g. \cite[Fact 1.6]{VZ3}) $g$ and $h$ are Lipschitz on $D$ and  $g|_D$, $h|_D$
 have Lipschitz convex extensions $\tilde g$, $\tilde h$ to all $X$. Thus we can set
 $\tilde f\coloneqq  \tilde g - \tilde h$.
	The assertion (ix) is a special case of \cite[Lemma 4.8.]{VZ} (``Mixing lemma'').
\end{proof}

\begin{lemma}\label{strict}
Let $F: (a,b) \to \R^d$ be a DC mapping and $x\in (a,b)$. Then the one-sided derivatives $F'_{\pm}(x)$ exist. Moreover
\begin{equation}\label{limder}
 \lim_{t\to x+} F'_{\pm}(t)=F'_+(x)\ \ \text{and}\ \       \lim_{t\to x-} F'_{\pm}(t)=F'_-(x)
\end{equation}
which implies that $F'_+(a)$ is the strict right derivative of $F$ at $x$, i.e.,
\begin{equation}\label{strd}
 \lim_{\substack{(y,z) \to (x,x)\\ y \neq z,\, y\geq x,\, z\geq x}}  \ \frac{F(z) - F(y)}{z-y} \ = \ F'_+(x).
 \end{equation}
\end{lemma}
\begin{proof}
The assertion easily follows from \cite[Theorem B, p. 325]{VZ2} and 
\cite[p.~329, Proposition 3.4~(i), (ii)]{VZ2}.
(Let us note that \cite{VZ2} works with Banach space valued $F$. In our case it is sufficient to observe that the assertion is easy if $n=1$ and $F$ is convex, and that the general case follows from this special case.)
\end{proof}

In the following, we will extensively work with DC surfaces in Euclidean spaces and {\it in 
 their subspaces}. So we will define DC surfaces in finite-dimensional Hilbert spaces only.
 The main notion for us is a $k$-dimensional DC surface in a $d$-dimensional Hilbert space
  which is given ``explicitely'' (i.e., as a ``graph'' of a DC mapping). From this reason
	 we use the term ``DC manifold''  for sets which are locally DC surfaces in our sense.
	
	In the rest of this subsection, the symbols $X$, $V$, $W$ always denote Hilbert spaces
	 of finite positive dimension.
	
\begin{definition}\label{plochy}
\begin{enumerate}
\item[(a)]
Let $d\coloneqq \dim X $, $A \subset X$ and $0<k<d$. We say that $A$ is a {\it DC surface} in $X$ of dimension $k$, if there exists a $k$-dimensional space $W \subset X$  and a Lipschitz DC mapping $\vf: W \to V\coloneqq  W^{\perp}$ such that $A = \{w+\vf(w): \ w \in W\}$. Then we will also say that $A$ is a DC  surface {\it associated with $V$}. A DC hypersurface (in $\R^d$) is a DC surface of dimension $d-1$.
\item[(b)]
We say that $\varnothing \neq A \subset \R^d$ is a  {\it DC} ({\it Lipschitz}) {\it manifold} of dimension $k$ ($0<k<d$), if for each $a \in A$ there exist  a $k$-dimensional vector space   $W \subset \R^d$, an open ball $U$ in $W$ and  a DC (Lipschitz, respectively) mapping $\vf:U \to W^{\perp}$ such that $P\coloneqq  \{w + \vf(w):\ w \in U\}$ is a relatively  open subset of $A$ and $a\in P$. 
\item[(c)]
For formal reasons, by a DC surface (resp.\ DC manifold) of dimension $k=d$ in $X$ we mean the whole space $X$ (resp.\ a nonempty open subset of $X$), and by a DC surface (resp.\ DC manifold) of dimension $k=0$ we mean a singleton (resp.\ a nonempty isolated set) in $X$.
\end{enumerate}
\end{definition}
	
\begin{remark}\label{uzpl}
\begin{enumerate}
\item [(i)]
It is easy to see that each DC surface  in $X$ is  closed in $X$, but a DC manifold in $X$
 can be non-closed.
\item[(ii)]
Each $k$-dimensional $C^{1,1}$ manifold in $X$ is a $k$-dimensional DC manifold.
\end{enumerate}
\end{remark}
		
We will need the following lemmas which  are straightforward consequences of	Lemma \ref{vldc}.
\begin{lemma}\label{sloz}
Let $d\coloneqq \dim X $,  $0<k<d$ and  $W\subset X$ be a $k$-dimensional space. Let $\vf: W \to W^{\perp}$ be a Lipschitz DC mapping
 and $P$ be a DC surface of dimension $k-1$ in $W$. Then $Q\coloneqq  \{w+ \vf(w):\ w \in P\}$
 is a DC surface of dimension $k-1$ in $X$.
\end{lemma}
\begin{proof}
The case $k=1$ is trivial. If $k>1$, then there exist a space $Z\subset W$ of dimensions $k-1$  and a Lipschitz DC mapping $\psi: Z \to Z^{\perp} \cap W$ such that 
 $P= \{z+ \psi(z):\ z \in Z\}$. 
 Then  $Q= \{z + \psi(z) + \vf(z + \psi(z)):\ z \in Z\}$. Note that also $\psi: Z \to Z^{\perp}$ is DC by
 Definition \ref{dc}~(ii). Consequently the
  mapping $\vf^*: Z \to Z^{\perp}$,\ $\vf^*(z)= \psi(z) + \vf(z + \psi(z))$ is Lipschitz and DC
 by Lemma \ref{vldc}~(i), (iv), (vi), and the assertion follows.
\end{proof}

\begin{lemma}\label{L:twoDCgraphs}
	Let $0<k<d$, $W$ and $Z$ be two $k$-dimensional subspaces of $\er^d$ and let $G\subset W$ and $H\subset Z$ be open.
	Suppose that there are two Lipschitz functions $f:G\to W^\bot$ and $g:H\to Z^\bot$ 
	such that $\{s+f(s):s\in G\}=\{t+g(t):t\in H\}\eqqcolon  S$ and suppose that $f$ is locally DC.
	Then $g$ is locally DC.
\end{lemma}
\begin{proof}
	Define $\phi:G\to\er^d$ by $\phi(s)=s+f(s)$ and $\psi:H\to\er^d$ by $\psi(t)=t+g(t)$.
	Then both $\phi^{-1}= \pi_W |_S$ and $\psi^{-1} = \pi_Z |_S$ are Lipschitz.
	So $F:G \to H$, $F\coloneqq  \psi^{-1} \circ \phi= \pi_Z \circ \phi$, is Lipschitz and also locally DC by Lemma \ref{vldc}~(i), (iii), (vi). Since $F^{-1}:H \to G$, $F^{-1}= \pi_W \circ \psi$ is Lipschitz, 
	Lemma \ref{vldc}~(vii) implies that $F^{-1}$ is locally DC, and so
		$g= \pi_{Z^{\perp}} \circ \phi \circ F^{-1}$ is locally DC as well.
\end{proof}

\subsection{WDC sets}
\begin{definition} \rm
Let $U\subset\R^d$ be open and $f:U\to\R$ be locally Lipschitz. A number $c\in \R$ is called a
\begin{enumerate}
\item[(i)] {\it regular value} of $f$ if $0\not\in\partial f(x)$ for all $x\in U$ such that $f(x)=c$;
\item[(ii)] {\it weakly regular value} of $f$ if whenever $x_i\to x$, $f(x_i)>c=f(x)$ and $u_i\in\partial f(x_i)$ for all $i\in\N$ then $\liminf_i|u_i|>0$.
\end{enumerate}
\end{definition}

\begin{remark} \label{rem_w_reg}\rm
\begin{enumerate}
\item[(i)] If $c$ is a regular value of $f$ (notice that $c$ need not be in the range of $f$)   then $c$ is also weakly regular, by \eqref{uzcl}.
\item[(ii)] Let $f$ be locally Lipschitz on $\rd$, $c\in f(\rd)$ and $\{f\leq c\}$ compact. Then, it is easy to see that $c$ is a weakly regular value of $f$ if and only if there exists an $\ep>0$ such that 
\begin{equation}  \label{E_reg}
|u|\geq\ep \text{ whenever } u\in\partial f(x) \text{ for some } x \text{ with } 0<\dist(x,\{f\leq c\})<\ep.
\end{equation}
\item[(iii)] If $f,c$ are as in (ii) and, moreover, $f$ is {\it proper} (i.e., $f^{-1}(K)$ is compact whenever $K\subset\R$ is compact), then $c$ is a weakly regular value of $f$ if and only if there exists an $\ep>0$ such that 
\begin{equation}  \label{E_reg_prop}
|u|\geq\ep \text{ whenever } u\in\partial f(x) \text{ for some } x \text{ with } c<f(x)<c+\ep.
\end{equation}
\end{enumerate}
\end{remark}

\begin{definition}[WDC set]  \rm
A set $A\subset\rd$ is called {\it WDC} if there exists a DC function $f:\rd\to [0,\infty)$ such that $A=f^{-1}\{0\}$ and $0$ is a weakly regular value of $f$. 
In such a case, we call $f$ a {\it DC aura} (for $A$).
\end{definition}

Notice that $\varnothing$ is a WDC set by our definition.

\begin{definition}[Locally WDC set]  \label{lwdc}\rm
A set $A\subset\rd$ is called {\it locally WDC} if for any point $a\in A$ there exists a WDC set $A^*\subset\R^d$ that agrees with $A$ on an open neighbourhood of $a$.
\end{definition}

\begin{remark}\label{2.2} \rm
\begin{enumerate}
\item[(i)] Equivalently, we can say that $A$ is WDC iff $A=f^{-1}((-\infty,c])$ for a DC function $f$ with weakly regular value $c$.
\item[(ii)] WDC sets were introduced in \cite{PR13} under the compactness assumption and the definition was extended to the Riemannian setting in \cite{FPR15}. Locally WDC sets were defined in \cite{PR13} with a formally stronger requirement of local agreement with compact WDC sets; nevertheless, it follows from Proposition~\ref{P_local} that both definitions are equivalent.
\item[(iii)] Auras were defined originally by Fu \cite{Fu94} as nonnegative {\it Monge-Amp\`ere} functions with weakly regular value $0$. 
Since any DC function is Monge-Amp\`ere (\cite[Theorem~1.1]{PR13}), a DC aura is an aura in Fu's sense.
We shall sometimes use the analogy and call a {\it Lipschitz aura} a nonnegative Lipschitz function with weakly regular value $0$ (though, of course, such a function need not be Monge-Amp\`ere).
\item[(iv)] It was shown in \cite{PR13} that compact {\it locally} WDC sets admit normal cycles and curvature measures. Normal cycles are defined locally, hence, the compactness assumption can be relaxed and both normal cycles and curvature measures can be defined even for {\it closed} locally WDC sets, and the main properties remain true (in particular, the local principal kinematic formula, cf.\ \cite{FPR15}, Theorem~B and \S5.1~(3)). Therefore, we are mainly interested in the (larger) class of closed locally WDC sets and its properties. On the other hand, it seems to be plausible that any closed locally WDC set is WDC (cf.\ \cite[Problem~10.2]{PR13}).
\item[(v)] Using the well-known fact that each Lipschitz convex function defined on a closed ball has a Lipschitz convex extension to the whole space, one can easily show that any {\it compact} WDC set $M$ admits a DC aura which is Lipschitz and, consequently,  satisfies \eqref{E_reg} for $c=0$ and some $\ep>0$. Since a positive multiple of a DC aura is again a DC aura, we can find a $1$-Lipschitz DC aura for $M$. Moreover, there even exists a $1$-Lipschitz DC aura for $M$ which is {\it proper} (indeed, consider the maximum of a $1$-Lipschitz DC aura for $M$ with the distance from an open ball containing $M$).
\end{enumerate}
\end{remark} 

We will consider also locally defined DC auras. Let $\varnothing\neq A\subset\R^d$ be closed. Given an open convex set $C\subset\R^d$, we say that a DC function $f:C\to [0,\infty)$ is a {\it DC aura in $C$} (for $A$) if $A\cap C=f^{-1}(\{0\})$ and $0$ is a weakly regular value of $f$. Such a locally defined DC aura cannot be extended in general to a DC aura (on $\R^d$), nevertheless, we show a related weaker result (Lemma~\ref{L_local}). As a corollary we obtain a characterization of locally WDC sets by ``local'' auras (Proposition~\ref{P_loc_aura}). Also, we show that any (locally) WDC set agrees locally with a {\it compact} WDC set (Proposition~\ref{P_local}). 

\begin{lemma}  \label{L_local}
Let $a\in\R^d$, $r\in (0,\infty]$ and let $f:B(a,r)\to[0,\infty)$ be a DC aura in $B(a,r)$ with $f(a)=0$ (here $B(a,\infty)=\R^d$). Then there exists $0<s<r$ and a compact WDC set $A^*\subset B(a,r)$ such that
$$A^*\cap B(a,s)=f^{-1}\{0\}\cap B(a,s).$$
If $r=\infty$ then $A^*$ with the above property can be found to any given $s>0$.
\end{lemma}

The proof is based on a corrected version of \cite[Proposition~7.3]{PR13} whose proof, unfortunately, contains a gap (cf.\ also the remark after Proposition~4.1 in \cite{FPR15}).  

Given an open convex set $C\subset\R^d$ and a DC aura $f$ in $C$, we denote
$$\widetilde{\nor}f\coloneqq \left\{\left(x,\tfrac{u}{|u|}\right):\, x\in C,\,f(x)=0,\, 0\neq u\in\partial f(x)\right\}.$$
(Note that, in general, $\widetilde{\nor}f$ is a larger set than $\nor(f,0)$ used in \cite{PR13} in the case $C=\R^d$.)
We shall say that two DC auras $f,g$ in $C$ {\it touch weakly} provided that there exists $(x,v)\in\widetilde{\nor}f$ with $(x,-v)\in\widetilde{\nor}g$. 
The corrected version of \cite[Proposition~7.3]{PR13} reads as follows.

\begin{lemma}   \label{L_f+g}
If $f,g$ are two DC auras in $C$ that do not touch weakly then $f+g$ is a DC aura in $C$ as well.
\end{lemma}

\begin{proof}
Obviously, $f+g$ is a DC function; it remains to verify the weak regularity of $0$.
We shall do this by contradiction: assume that $f(x)=g(x)=0$, $f(x_i)+g(x_i)>0$, $x_i\to x$, $w_i\in\partial (f+g)(x_i)$, but $w_i\to 0$.
Since $\partial(f+g)(x)\subset\partial f(x)+\partial g(x)$ (see  \cite[Proposition 2.3.3]{Cl90}), there exist vectors $u_i\in\partial f(x_i)$ and $v_i\in\partial g(x_i)$ with $w_i=u_i+v_i$. We can assume without loss of generality that say $f(x_i)>0$ for all $i$. Then, by the weak regularity of $0$ for $f$, we can assume that there exists an $\ep>0$ such that $|u_i|\geq\ep$ for all $i$. Turning to a subsequence if necessary, we may assume that there exist $\lim_iu_i\eqqcolon u$ and $\lim_iv_i\eqqcolon v$. So $u\neq 0$ and  $v= -u$.  Since $\graph\partial f$ and $\graph\partial g$ are closed  by \eqref{uzcl}, we have $u\in\partial f(x)$ and $v\in\partial g(x)$. So, obviously, $f,g$ 
 touch weakly, which is a contradiction.
\end{proof}

\begin{proof}[Proof of Lemma~\ref{L_local}]
We can assume without loss of generality that $a=0$.
First we claim that the set
$$\{(v,x\cdot v)\in S^{d-1}\times\R:\, (x,-v)\in\widetilde{\nor}f\}$$
has $d$-dimensional measure zero. This is, in fact, Proposition~7.1 from \cite{PR13} with a restricted domain of $f$ and $\nor(f,0)$ replaced by $\widetilde{\nor}f$, and the proof given in \cite{PR13} works also in our setting. Further, let us say that a $k$-dimensional affine subspace $F$ of $\R^d$ is {\it weakly tangent} to the DC aura $f$ in $C$ if there exists $(x,v)\in\widetilde{\nor}f$ such that $x\in F$ and $v\perp F$. As in Lemma~8.4 of \cite{PR13}, it can be shown that for any $0\leq k\leq d-1$, the set of affine $k$-subspaces weakly tangent to a DC aura $f$ in $C$ has measure zero, with respect to the motion invariant measure. 
(In fact, Lemma~8.4 in \cite{PR13} is stated only for $1\leq k\leq d-1$, but the proof given there works for $k=0$ as well; moreover, the statement with $k=0$ reads that $\partial f^{-1}\{0\}$ has $d$-dimensional measure zero, which clearly follows from \cite[Proposition~7.1]{PR13}.)
Hence, by \cite[Lemma~8.5]{PR13}, there exist linearly independent unit vectors $v_1,\ldots,v_d$ in $\R^d$ and dense sets $T_1,\ldots,T_d\subset \R$ such that the affine $k$-subspace
$$\{x\in\R^d:\, x\cdot v_{i_1}= t_{i_1},\ldots,x\cdot v_{i_{d-k}}= t_{i_{d-k}}\}$$
is not weakly tangent to $f$ whenever $t_i\in T_i$, $i=1\ldots,d$, $0\leq k\leq d-1$ and $1\leq i_1<\dots <i_{d-k}\leq d$.
The set $A^*$ can be found now as the intersection of $f^{-1}\{0\}$ with a polytope
$$P\coloneqq \{x\in\R^d:\, -s_1\leq x\cdot v_1\leq t_1,\ldots, -s_d\leq x\cdot v_d\leq t_d\}$$
with $s_i, t_i >0$ and $-s_i,t_i\in T_i$, $i=1,\ldots,d$. (Indeed, $P$ contains the origin in its interior and $P\subset B(0,r)$ if $s_i,t_i>0$ are small enough. If, on the other hand, $r=\infty$ we can choose $s_i,t_i$ large enough so that $P$ contains $B(0,s)$ for any given $s>0$.)
The distance function $g:x\mapsto\dist (x,P)$ is clearly a DC aura and its restriction to $B(0,r)$ does not weakly touch $f$. Hence, an application of Lemma~\ref{L_f+g} guarantees that $f+g$ is a DC aura in $B(0,r)$. Let $0<s<r$ be such that $P\subset B(0,s)$. By Lemma~\ref{vldc}~(viii), there exists a DC function $h:\R^d\to[0,\infty)$ such that $h|_{B(0,s)}=(f+g)|_{B(0,s)}$. Denoting $p\coloneqq \dist(\cdot,B(0,s))$, the function $q\coloneqq \max\{h,p\}$ is a DC aura for $A^*$ (note that $p^{-1}\{0\}\subset P$).
\end{proof}

\begin{proposition}  \label{P_loc_aura}
Let $\varnothing\neq A\subset\R^d$ be closed. Then, the following are equivalent.

\begin{enumerate}
\item[{\rm (i)}] $A$ is locally WDC.
\item[{\rm (ii)}] For any $a\in A$ there exist an $r>0$ and a DC aura $f: B(a,r)\to [0,\infty)$ in $B(a,r)$ for $A$ (in particular, $f^{-1}\{0\}=A\cap B(a,r)$).
\end{enumerate}
\end{proposition}

\begin{proof}
Since any DC aura is also a DC aura in $C$ for any $C\subset\R^d$ open convex, the implication (i) $\implies$ (ii) is obvious. To show the other implication, let $A$ be a closed set fulfilling (ii), and let $a\in A$ be given. By (ii), there exists a DC $B(a,r)$-aura $f$ for some $r>0$. Using Lemma~\ref{L_local}, we can find a (compact) WDC set $A^*$ which agrees with $A$ on some neighbourhood of $a$. Thus, $A$ is locally WDC.
\end{proof}   

\begin{proposition}[Localization of WDC sets]  \label{P_local}
Let $A\subset\R^d$ be a WDC set, $a\in A$ and $s>0$. Then there exists a compact WCD set $A^*\subset A$ such that 
$A^*\cap B(a,s)=A\cap B(a,s)$.
\end{proposition}

\begin{proof}
Apply Lemma~\ref{L_local} with $r=\infty$.
\end{proof}

\section{Deformation lemma for Lipschitz mappings}
Let $X$ be topological space and $Y\subset X$. A continuous mapping $H:X\times[0,1]\to X$ is a {\em deformation retraction} of $X$ onto $Y$ if $H(x,0)=x$ whenever $x\in X$, $H(y,t)=y$ whenever $y\in Y$ and $H(X\times\{1\})=Y$ (cf.\ \cite[p.~2]{Hat02}).

The following lemma extends \cite[Lemma~4.1]{CC01}, removing the smoothness assumption. Note that the notion of a proper mapping is given in Remark~\ref{rem_w_reg}~(iii).

\begin{lemma} \label{DL}
Let $f:\rd\to\R$ be Lipschitz and proper, $M\coloneqq \{f\leq 0\}\neq\varnothing$ compact and assume that $0$ is a weakly regular value of $f$. Then there exist $0<\ep<1$ and a continuous mapping 
$H:\{f<\ep\}\times[0,1]\to \{f<\ep\}$ such that
\begin{enumerate}
\item[{\rm (i)}] $|v|\geq\ep$ whenever $v\in\partial f(x)$ and $0<f(x)<\ep$,
\item[{\rm (ii)}] if $0<f(x)<\ep$ then $H(x,0)=x$ and $H(x,1)\in\partial M$,
\item[{\rm (iii)}] if $x\in M$ then $H(x,t)=x$ for all $t\in[0,1]$,
\item[{\rm (iv)}] for any $(x,t)\in \{0\leq f<\ep\}\times[0,1]$,
$$\frac{\ep}{2}|H(x,t)-H(x,1)|\leq f(H(x,t)),$$
\item[{\rm (v)}] if $f(z)<\ep$ then
$$\frac{\ep}2\,\dist(z,M)\leq f(z)\leq(\lip f)\,\dist(z,M),$$
\item[{\rm (vi)}] for any $0<r<\ep$ and $y\in\partial M$ there exists $x$ with $f(x)=r$ and $H(x,1)=y$,
\item[{\rm (vii)}] for any $0<r<\ep$, the restriction of $H$ to $\{f\leq r\}\times [0,1]$ is a deformation retraction of $\{f\leq r\}$ onto $M$. 
\end{enumerate}
\end{lemma}

\begin{proof}
Since $0$ is a weakly regular value of $f$, there exists an $\ep>0$ such that (i) holds, see Remark~\ref{rem_w_reg}~(3). It is not difficult to see that we can take $\ep\in (0,1)$. Denote $U_\ep\coloneqq \{x:\, 0<f(x)<\ep\}$.
As the first step we show that there exists a bounded $C^\infty$ mapping $F:U_\ep \to\rd$ such that
\begin{equation}  \label{E2}
\langle F(x),v\rangle > \ep\text{ whenever }x\in U_\ep \text{ and }v\in\partial f(x).
\end{equation}
To show this, we slightly modify the construction from the proof of \cite[Lemma~3.1]{BC90}. Given $u\in\sph$, denote $V(u)\coloneqq \{v\in\rd:\, v\cdot u>\frac\ep 2\}$ and $G(u)\coloneqq \{x\in U_\ep :\, \partial f(x)\subset V(u)\}$.

Let $\mathcal K$ be the space of all nonempty convex compact subsets of $\rd$ equipped with the Hausdorff metric, and let ${\mathcal V},{\mathcal V}(u)$ ($u\in\sph$) denote its subsets of all $K\in{\mathcal K}$ contained in $\{v\in\rd:\, \ep\leq|v|\leq\lip f\}$, $V(u)$, respectively. Since ${\mathcal V}$ is compact and covered by the open sets ${\mathcal V}(u)$, $u\in\sph$, there exist finitely many unit vectors $u_1,\ldots,u_k$ such that
${\mathcal V}\subset{\mathcal V}(u_1)\cup\dots\cup{\mathcal V}(u_k)$. Note that, by (i), $\partial f(x)\in{\mathcal V}$ whenever $x\in U_\ep $. 
Thus, also $U_\ep \subset G(u_1)\cup\dots\cup G(u_k)$.
Since $x\mapsto\partial f(x)$ is upper semicontinuous (see \eqref{uzcl}), the sets $G(u_i)$ are open, and there exists a $C^\infty$ partition of unity $1=\sum_{i=1}^k\gamma_i$ on $U_\ep $ subordinated to the open cover $G(u_1),\dots,G(u_k)$. The mapping

$$F(x)=2\sum_{i=1}^k\gamma_i(x)u_i$$
then satisfies \eqref{E2}; note that $\ep\leq|F(x)|\leq 2$ for all $x\in U_\ep $.

Consider now the differential equation
\begin{equation}  \label{DE}
x'(t)=-F(x(t)),\quad x(0)=x,
\end{equation}
for $x\in U_\ep $. Since $F$ is $C^\infty$ smooth, it is locally Lipschitz in $U_\ep $ and, hence, there exists a unique maximal solution 
$\varphi_x:I_x\to U_\ep$ for any $x\in U_\ep$. Denote $\tau(x)\coloneqq \sup I_x$; we shall find an upper bound for $\tau(x)$. Consider the function
$g_x\coloneqq f\circ\varphi_x$ which is clearly Lipschitz and, by the chain rule for Lipschitz functions (see \cite[Theorem~2.3.9]{Cl90}), its Clarke gradient satisfies
\begin{eqnarray}
\partial g_x(t)&\subset&\{ \langle \varphi_x'(t),v\rangle:\ v\in\partial f(\varphi_x(t))\}\nonumber\\
&=&\{ -\langle F(\varphi_x(t)),v\rangle:\ v\in\partial f(\varphi_x(t))\}\nonumber\\
&\subset&[-2\lip f,-\ep]; \label{grad_g}
\end{eqnarray}
we have used \eqref{E2} and $|F|\leq 2$ in the last inclusion. The mean value theorem for Lipschitz functions \cite[Theorem~2.3.7]{Cl90} yields for $t\in I_x$
$$g_x(t)\leq g_x(0)-\ep t=f(x)-\ep t,\quad t\geq 0,$$
and, since clearly $g_x(t)>0$, we obtain $t<f(x)/\ep$, hence,
\begin{equation} \label{tau}
\tau(x)\leq \ep^{-1}f(x)\leq \lip f
\end{equation}
since $x\in U_\ep $.

The mapping $\varphi_x$ is Lipschitz (the norm of its derivative is bounded by $2$, see \eqref{DE}) on $I_x$ and, hence, it has Lipschitz extension to $I_x\cup\{\tau(x)\}\supset [0,\tau(x)]$. Note that necessarily 
$$g_x(\tau(x))=0,\quad x\in U_\ep .$$
Indeed, $g_x(\tau(x))\geq 0$ by continuity, and if $g_x(\tau(x))>0$ 
we would obtain a contradiction with the fact that no maximal solution can end inside $U_\ep $, see e.g. \cite{HS74}, Theorem in Ch.~8, \S5.

We shall show now that $\tau$ is continuous. Fix an $x\in U_\ep$ and $\delta>0$ and let $x_i\to x$. Then $f(\varphi_x(\tau(x)-\delta))>0$, and, since the solutions of \eqref{DE} depend continuously on the initial condition $x$ (see \cite[Ch.~8, \S6]{HS74}), also $f(\varphi_{x_i}(\tau(x)-\delta))>0$ for sufficiently large $i$, hence, $\tau(x_i)>\tau(x)-\delta$. Since $\delta>0$ can be arbitrarily small, we get $\liminf_i\tau(x_i)\geq\tau(x)$. For the other inequality, note that $g_x(\tau(x)-\delta)<2(\lip f)\delta$ by \eqref{grad_g} and, again by the continuity in initial conditions, we have $g_{x_i}(\tau(x)-\delta)<3(\lip f)\delta$ for sufficiently large $i$. Using now the other bound from \eqref{grad_g}, we get that $g_{x_i}(\tau(x)-\delta+\eta)<0$ provided that $\eta=\frac{3\lip f}{\ep}\delta$ and $\tau(x)-\delta+\eta$ lies in the domain of $g_{x_i}$. But $g_{x_i}\geq 0$ on its domain, which implies that $\tau(x_i)\leq \tau(x)-\delta+\eta$, and since $\delta$ (and, hence, also $\eta$) can be arbitrarily small, we get $\limsup_i\tau(x_i)\leq\tau(x)$. Thus we have proved that $\tau$ is continuous.

We define now
$$H(x,t)\coloneqq \begin{cases} \varphi_x(t):& x\in U_\ep,\, 0\leq t\leq\tau(x)\\
                 \varphi_x(\tau(x)):& x\in U_\ep,\, \tau(x)\leq t\leq 1\\
								 x:& f(x)\leq 0
					\end{cases}
$$
The continuity of $H$ can be shown similarly as in \cite{CC01}, the idea is as follows: Note that the mapping $H(x,\cdot)$ is Lipschitz with constant $2$ for each $x$. Thus, for the continuity in both variables, is is enough to show that $H(\cdot,t)$ is continuous for each $t$. Let $t\in [0,1]$ and $x\in\{ f<\ep\}$ be given; the continuity of $H(\cdot,t)$ at $x$ can be seen by distinguishing several cases. (a) If $f(x)>0$ and $t<\tau(x)$ then we can use \cite[Ch.~8, \S6]{HS74}, as above. (b) If $f(x)>0$ and $t=\tau(x)$, choose an arbitrary $\omega>0$. Now consider an $s<t$ and note that by (a), for $y$ sufficiently close to $x$, $|H(y,s)-H(x,s)|<\omega/2$. By the Lipschitz property of $H$ in the second variable, we have $$|H(y,t)-H(x,t)|\leq \frac\omega 2+2\cdot 2(t-s)<\omega$$
for $t-s$ small enough. (c) If $f(x)>0$ and $t>\tau(x)$ then $H(x,t)=H(x,\tau(x))$ and also $H(y,t)=H(y,\tau(x))$ for $y$ close to $x$, so our assertion follows from (b). (d) If $f(x)\leq 0$ then $H(x,t)=x$. Consider an $y$ close to $x$; then either $f(y)\leq 0$ and, hence $H(y,t)=y$ is close to $x$, or $f(y)>0$ and $\tau(y)$ is small due to \eqref{tau}, which implies that $|y-H(y,t)|\leq 2\tau(y)$ is small as well and the proof of continuity is finished.

The properties (ii) and (iii) follow immediately from the above considerations. We shall verify (iv). If $x\in f^{-1}((-\infty,0])$ or $t\geq\tau(x)$ then the inequality is obvious. Take now $x\in U_\ep$ and $0\leq t\leq\tau(x)$ and note that, by the mean value theorem for the Lipschitz function $g_x$ and \eqref{grad_g},
$$f(H(x,t))=g_x(t)\geq g_x(\tau(x))+(\tau(x)-t)\ep=(\tau(x)-t)\ep$$
which, together with the fact that $\varphi_x$ is $2$-Lipschitz, implies (iv). 

Using (iv) with $x=z$ and $t=0$ and the Lipschitz property of $f$, we obtain (v). 

In order to prove (vi), take any $y\in\partial M$ and a sequence $y_i\to y$ with $f(y_i)>0$. Since clearly $H(y,1)=y$, we have $H(y_i,1)\to y$ by continuity. If $\varphi_x:I_x\to U_\ep$ is any maximal solution of \eqref{DE}, we have shown that $g_x(\sup I_x)=0$, and similarly it follows that $g_x(\inf I_x)=\ep$ (since if $g_x(\inf I_x)\in (0,\ep)$ we would get a contradiction with the maximality). Thus, any maximal solution hits all level sets $\{f=r\}$ with $0<r<\ep$. Hence, there exist $t_i\in I_{y_i}$ and $x_i\in\{f=r\}$ with $x_i=\varphi_{y_i}(t_i)$. Since the equation \eqref{DE} is autonomous, we have $\varphi_{y_i}(t_i+t)=\varphi_{x_i}(t)$ whenever one side is defined (see \cite{HS74}, Theorem~1 in Ch.~8, \S7)  and we obtain easily that $H(x_i,1)=H(y_i,1)\to y$. It follows that any accumulation point $x$ of $(x_i)$ satisfies $f(x)=r$ and $H(x,1)=y$, by continuity of $H$. 

Property (vii) follows using the fact that $f(H(x,t))$ is decreasing in $t$. 
\end{proof}

\begin{corollary}\label{Cor:propertiesOfRetracts}
Let $f,M$ and $\ep>0$ be as in Lemma~\ref{DL}. Then, 
\begin{enumerate}
\item $\chi(M)=\chi(\{f\leq r\})$, $0\leq r<\ep$,
\item $M$ is locally contractible (hence, locally arcwise connected),
\item both $M$ and $\rd\setminus M$ have finitely many connected components.
\end{enumerate}
In particular, (1), (2) and (3) hold if $M$ is a compact WDC set and $f$ its proper DC aura (cf.\ Remark~\ref{2.2}~(v)).
\end{corollary}

\begin{proof}
Assertion (1) follows from (vii) and the well-known fact that any deformation retraction is a homotopy equivalence, see \cite[Ch.~III.5]{Dold80}.

Lemma~\ref{DL} implies that $M$ is a retract of its open neighbourhood, hence, a Euclidean neighbourhood retract, and such spaces are known to be locally contractible, see e.g.\ \cite[IV. (3.1), V. (2.6)]{Bor67}; this proves (2).

Finally, (3) follows from the well-known fact that the homology groups of a Euclidean neighbourhood retract are finitely generated (see, e.g., \cite[Corollary~A.8]{Hat02}).
\end{proof}

\begin{notation}\label{fixaury}
If $\varnothing\neq M \subset \R^d$ is a compact WDC set, we always fix $f=f_M$, $\ep = \ep_M$ and $H=H_M$
 in the following way. First we choose  a proper $1$-Lipschitz DC aura $f$ for $M$ (see Remark~\ref{2.2}~(v))
 and then choose  $0<\ep<1$ and $H$ as in Lemma~\ref{DL}. Note  that  $H(x,t)$ is defined
 for all  $x \in B(M,\ep)\coloneqq \{y:\, \dist(y,M)<\ep\}$, $t \in [0,1]$, and that $|u| \geq \ep$ whenever 
 $x \in B(M,\ep)$  and  $u \in \partial f(x)$. Indeed, since $f$ is $1$-Lipschitz,
 we have $B(M,\ep) \subset \{f < \ep\}$.
\end{notation}

\begin{lemma}\label{L:connecting}
Let $\varnothing\neq M \subset \R^d$ be a  compact WDC set. Let  $0<\ep=\ep_M< 1$  be as in Notation \ref{fixaury}. Suppose that $0< \delta < \ep$,
 $x,y\in \partial M$ 
 and $\vf:[0,1]\to \R^d$ is a continuous curve for which 
$$ \vf(0)=x,\ \vf(1)=y,\ \vf((0,1)) \subset \R^d \setminus M\ \ \text{and}\ \  \diam(\vf([0,1])) \leq \delta.$$
 Then there exists a continuous $\kappa: [0,1] \to \partial M$ for which
$$\kappa(0)=x,\ \kappa(1)=y,\ \ \text{and}\ \ \diam(\kappa([0,1])) < \frac{6\delta}{\ep}.$$
\end{lemma}

\begin{proof}
 Let $f=f_M$ and $H= H_M$ be as in Notation \ref{fixaury}.
Note that the function $\kappa(u)= H(\vf(u),1),\ u \in [0,1],$ is continuous with the range in $\partial M$.
By Lemma \ref{DL}~(iii), (vi) we obtain, for each  $0\leq u \leq 1$,   
\begin{eqnarray}
 |\kappa(u) -\vf(u)|= |H(\vf(u),1)-H(\vf(u),0)| &\leq& \frac{2}{\ep} f(\vf(u))\nonumber \\  &\leq& \frac{2}{\ep} \dist(\vf(u),M) \leq \frac{2}{\ep} \delta  \label{zdere}.
\end{eqnarray}
Therefore, for each $u \in [0,1]$, 
$$|\kappa(u) -x| \leq |\kappa(u)- \vf(u)| + |\vf(u)-x|  \leq \frac{2}{\ep}\delta+ \delta < \frac{3\delta}{\ep}. $$
 Consequently,  $\diam(\kappa([0,1])) < \frac{6\delta}{\ep}$ .
\end{proof}

Consider the situation of Lemma~\ref{DL}. If the mapping $H$ would be, moreover, Lipschitz, we would get that $\partial M$ is $(d-1)$-rectifiable (indeed, if $0<r<\ep$ then $\partial M=H(\cdot,1)(\{ f=r\})$ by Lemma~\ref{DL}~(vi) and the level set $\{f=r\}$ is a Lipschitz manifold of dimension $d-1$ since $r$ is a regular value of $f$). 
The following example shows, however, that the set $M$ fulfilling the assumptions of Lemma~\ref{DL} need not be $(d-1)$-rectifiable. Hence, the mapping $H$ does not always exist Lipschitz (cf. Question~\ref{Q:LR}).

\begin{example}\label{E:retractionNotLipschitz}
There exists a compact set $K\subset\R^2$ with empty interior and Hausdorff dimension greater than $1$ which admits a Lipschitz aura.
\end{example}

\begin{proof}
Fix an angle $\alpha\in (0,\frac{\pi}8)$ and consider the points
$$
V^{\pm}=\left(\pm\frac{1}{2},0\right),\quad V=\left(0,\frac 12 \tan\alpha\right),\quad 
B^{\pm}=\left(\pm\left(\frac{1}{2}-\frac 1{4\cos^2\alpha}\right),0\right)
$$
and triangles
$$
H=\conv\{V^-,V^+,V\}\quad\text{and}\quad H^{\pm}=\conv\{V^{\pm},B^{\pm},V\}.
$$

We further denote by $\gamma$ the size of the convex angle $\angle(B^-VB^+)$ and note that
$\gamma = \pi-4\alpha\in(\frac{\pi}{2},\pi)$.

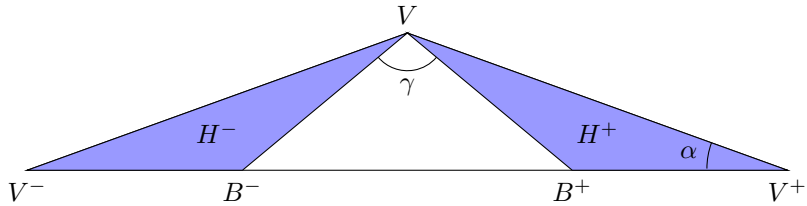
\begin{tikzpicture}
\draw (-5,0)node[anchor=north] {$V^-$}  -- (5,0) node[anchor=north] {$V^+$} -- (0,1.82) node[anchor=south] {$V$} -- cycle;
\filldraw[fill=blue!40!white, draw=black] (-5,0) -- (-2.1688,0) node[anchor=north] {$B^-$} -- (0,1.82) -- cycle;
\filldraw[fill=blue!40!white, draw=black] (5,0) -- (2.1688,0) node[anchor=north] {$B^+$} -- (0,1.82) -- cycle;
\node at (-2.5,0.5) {$H^-$};
\node at (2.5,0.5) {$H^+$};
\draw (4,0.364) arc (160:180:1);
\node at (3.7,0.24) {$\alpha$};
\draw (-0.383,1.4986) arc (220:320:0.5);

\node at (0,1.1) {$\gamma$};
\end{tikzpicture} 

It is not difficult to verify that all the three triangles $H,H^-,H^+$ are similar. Let $\phi_+,\phi_-$ be the similarities mapping $H$ onto $H^+,H^-$ and keeping $V^+,V^-$ fixed, respectively. Both $\phi_+,\phi_-$ are contracting similarities with the same coefficient $a\coloneqq \frac{1}{2\cos\alpha}\in (\frac 12,1)$ and, hence, there exists a unique self-similar set $K\subset\R^2$ satisfying $K=\phi_{-}(K)\cup\phi_{+}(K)$, see \cite[Theorem~9.1]{Falconer}. (The choice $\alpha=\frac{\pi}{6}$ at the beginning would yield the well-known von Koch curve.) Moreover, the open set condition clearly holds (with open set $\INt H$), hence, the Hausdorff dimension of $K$ is $\ln\frac 12 / \ln a\in (1,2)$, see \cite[Theorem~9.3]{Falconer}, and elementary arguments lead to
$$V\in K\subset H^+\cup H^-.$$
We will show that the distance function $d_K: x\mapsto\dist(x,K)$, $x\in\R^2$, is a proper Lipschitz aura for $K$. $d_K$ is clearly proper and $1$-Lipschitz. We further show that
\begin{equation}  \label{regul}
|v|\geq -\cos\gamma\text{ whenever }x\in\R^2\setminus K\text{ and }v\in\partial d_K(x).
\end{equation}
This will imply that $0$ is a weakly regular value of $d_K$ and the proof will be complete.

Let $x,v$ be as in \eqref{regul}. Using \cite[Lemma~4.2]{Fu85}, we know that
$$v\in\conv \left\{\frac{x-y}{|x-y|}:\, y\in\Pi_K(x)\right\}\eqqcolon C_x,$$
where $\Pi_K(x)$ denotes the metric projection of $x$ to $K$ (i.e., the set of all points of $K$ lying in distance $d_K(x)$ from $x$). 
Thus, $|v|\geq\dist(0,C_x)$. If $\Pi_K(x)$ is a singleton then $|v|=1$ and \eqref{regul} holds. In the sequel, we will assume that $\Pi_K(x)$ has at least two points.

Assume that $\Pi_K(x)$ is contained in one of the triangles $H^+,H^-$, say in $H^-$.
Then, by the self-similarity of $K$, the preimage $x_1\coloneqq \phi_-^{-1}(x)$
satisfies $\Pi_K(x_1)=\phi_-^{-1}(\Pi_K(x))$ and, consequently, $C_{x_1}$ agrees with $C_x$ up to a linear isometry (in particular, $\dist(0,C_{x_1})=\dist(0,C_x)$).
If again $\Pi_K(x_1)$ is contained in $H^-$ or $H^+$, we iterate the same procedure until, after a finite number of $n$ steps, we get 
$x_n\coloneqq \phi_n^{-1}\circ\dots\circ\phi_1^{-1}(x)$, $\dist(0,C_{x_n})=\dist(0,C_x)$, 
$\Pi_K(x_n)\cap (H^+\setminus\{V\})\neq\varnothing$ and $\Pi_K(x_n)\cap (H^-\setminus\{V\})\neq\varnothing$. (We use the fact that the diameter of $\Pi_K(x_n)$ is positive and is increased by factor $a^{-1}$ in each step.) We shall write $x$ instead of $x_n$ in the sequel.

The open ball $B\coloneqq B(x,d_K(x))$ does not hit $K$, hence, $V\not\in B$ and let $V'$ denote the intersection point of the segment $\overline{x,V}$ with $\partial B$. We observe that $x$ lies in the interior of the (convex) angle $\angle(B^-VB^+)$ (it follows from the facts that $\gamma$ is an obtuse angle, $\overline{B}$ intersects both $H^-\setminus\{V\}$ and $H^+\setminus\{V\}$, and $V\not\in B$). 
Let $y_-\in\Pi_K(x)\cap (H^-\setminus\{V\})$ and $y_+\in\Pi_K(x)\cap (H^+\setminus\{V\})$ be such that the angles $\beta_-\coloneqq \angle(Vxy_-)$ and $\beta_+\coloneqq \angle(Vxy_+)$ are maximal (recall that $\Pi_K(x)$ is a compact set). Thus, denoting $\beta\coloneqq \beta_-+\beta_+$, $2\pi-\beta$ is the central angle corresponding to the inscribed angle $\angle(y_-V'y_+)$ and we have 
$$\frac{2\pi -\beta}{2}=\angle(y_-V'y_+)>\angle(y_-Vy_+)\geq\gamma.$$
Since $\gamma>\frac{\pi}{2}$, we obtain $\beta<\pi$ and
$$\dist(0,C_x)=\frac{\dist(x,\overline{y_-,y_+})}{|x-y_-|}=\cos\frac\beta 2\geq-\cos\gamma$$
(note that the whole set $\Pi_K(x)$ lies on the shorter arc of $\partial B$ with endpoints $y_-$, $y_+$). Thus \eqref{regul} follows.
\end{proof}

\section{Gauss-Bonnet formula}  \label{Sec_GB}
We recall that a locally Lipschitz function $f:\rd\to\R$ is {\it Monge-Amp\`ere} if there exists a (necessarily unique) $d$-dimensional integral current without boundary $\bD f$ on $\rd\times\rd$ which annihilates the symplectic $2$-form ($\bD f\llc\omega=0$), its support has bounded first component and for any $g\in C^\infty_c(\rd\times\rd)$,
$$\bD f(g\cdot (\pi_0)^{\#}\Omega_d)=\int_{\rd}g(x,\nabla f(x))\, dx,$$
where $\pi_0:(x,y)\mapsto x$ is the first component projection, $\Omega_d$ is the volume form in $\rd$ and $\nabla f(x)$ the gradient of $f$ at $x$ (which exists Lebesgue-almost everywhere by the local lipschitzness). The support of $\bD f$ is contained in $\{(x,u):\, u\in\partial f(x)\}$. See \cite[Sect.~2]{Fu89} or \cite[Def.~5.1]{PR13}.

Let $f:\rd\to \R$ be Monge-Amp\`ere and let $r$ be a weakly regular value of $f$. Then we define the integral current
$$N(f,r)\coloneqq -\nu_{\#}\partial(\bD f \llc(f\circ\pi_0)^{-1}(r,\infty)),$$
where 
$$\nu:(x,y)\mapsto (x,y/|y|),\quad (x,y)\in\rd\times(\rd\setminus\{0\})$$
is the spherical projection. Note that the assumption that $r$ is a weakly regular value guarantees that the support of $\partial(\bD f \llc(f\circ\pi_0)^{-1}(r,\infty))$ is contained in the domain of $\nu$, thus the push-forward $\nu_{\#}$ is well defined. We also remark that $N(f,r)=\lim_{s\to r_+}N(f,s)$ and if $r$ is a regular value of $f$ then
$$N(f,r)=\nu_{\#}\langle \bD f,f\circ\pi_0,r\rangle,$$
see \cite[\S4.2.1, 4.3.4]{Fe69}. In what follows, we denote by $\pi_1:(x,y)\mapsto y$ the second component projection.
 
\begin{proposition}  \label{P_GB}
Let $f:\rd\to\R$ be proper, Lipschitz and Monge-Amp\`ere, and assume that $0$ is a weakly regular value of $f$. Then 
\begin{equation} \label{GB}
N(f,0)(\varphi_0)=\chi(\{f\leq 0\}),
\end{equation}
where $\varphi_0$ is the Gauss form, i.e., the differential form of order $d-1$ on $\rd\times\rd$ given as
$$\varphi_0(x,u)\coloneqq (d\omega_d)^{-1}(\pi_1)^{\#}(u\lrc\Omega_d),$$
$\omega_d=\pi^{d/2}/\Gamma(1+d/2)$ and $\chi$ is the Euler-Poincar\'e characteristic.
\end{proposition}

\begin{remark}
Since any DC function is Monge-Amp\`ere (see \cite[Theorem~1.1]{PR13}), Proposition~\ref{P_GB} applies whenever $f$ is a proper Lipschitz DC aura for a compact WDC set.
\end{remark}

\begin{proof}
Let $\ep>0$, $U_\ep$ and $F:U_\ep\to\rd$ be as in the proof of Lemma~\ref{DL}. We define
$$u:x\mapsto -\frac{F(x)}{|F(x)|},\quad x\in U_\ep$$
(recall that $|F|\geq\ep/2>0$),
$$V\coloneqq \{(x,y)\in U_\ep\times\rd:\, 0\not\in\overline{u(x),v}\}$$
and
\begin{eqnarray*}
h&:& [0,1]\times V \to \R^d\times S^{d-1}\\
   &&  (t,x,y)\mapsto\nu(x,(1-t)y+tu(x))
\end{eqnarray*}
(due to the definition of the domain $V$, $h$ is well defined and smooth). Note that $h(0,x,y)=\nu(x,y)$ and $h(1,x,y)=(x,u(x))$, $(x,y)\in V$.

In the following, we shall use the notation $[0,1]$ for both the closed unit interval and the $1$-dimensional current given by Lebesgue integration along it with natural orientation. We denote 
$$T\coloneqq h_{\#}([0,1]\times(\nu_{\#}(\bD f\llc U_\ep)));$$
$T$ is a $(d+1)$-dimensional current in $\rd\times\rd$ (for the definition of the cartesian product of currents, see \cite[\S4.1.8]{Fe69}). 
The boundary of $T$ can be computed (cf.\ \cite[\S7.4.3]{KP08})
\begin{eqnarray*}
\partial T&=&\partial h_{\#}([0,1]\times(\nu_{\#}(\bD f\llc U_\ep)))\\
&=&h_{\#}\partial([0,1]\times(\nu_{\#}(\bD f\llc U_\ep)))\\
&=&h_{\#}[(\delta_1-\delta_0)\times(\nu_{\#}(\bD f\llc U_\ep))]-
   h_{\#}[[0,1]\times\partial(\nu_{\#}(\bD f\llc U_\ep))]\\
&=&(\tilde{u}\circ\pi_0)_{\#}(\nu_{\#}(\bD f\llc U_\ep))-\nu_{\#}(\bD f\llc U_\ep)-
   h_{\#}[[0,1]\times\partial(\nu_{\#}(\bD f\llc U_\ep))],	
\end{eqnarray*}
where $\delta_0,\delta_1$ denote the $0$-currents corresponding to the Dirac measure at $0,1$, respectively, and $\tilde{u}(x)\coloneqq (x,u(x))$.
Since $\bD f$ has no boundary, the boundary of $\bD f\llc U_\ep$ is supported in $\{x:\, f(x)\in\{0,\ep\}\}$. 
The slice $\langle T,f\circ\pi_0,r\rangle$ is defined for any $0<r<\ep$ and its boundary satisfies (see \cite[p.~437]{Fe69})
$$\partial\langle T,f\circ\pi_0,r\rangle=-\langle\partial T,f\circ\pi_0,r\rangle.$$
Thus, for $0<r<\ep$ we obtain
\begin{eqnarray*}
\partial\langle T,f\circ\pi_0,r\rangle 
&=&	-\langle(\tilde{u}\circ\pi_0)_{\#}(\nu_{\#}\bD f),f\circ\pi_0,r\rangle +
\langle \nu_{\#}(\bD f\llc U_\ep),f\circ\pi_0,r\rangle\\
&=&	-\tilde{u}_{\#}\langle(\pi_0)_{\#}(\bD f),f,r\rangle + N(f,r)\\
&=&	-\tilde{u}_{\#}\langle\Ha^d\wedge e_{1\dots d},f,r\rangle + N(f,r),
\end{eqnarray*}
where $e_{1\dots d}=e_1\wedge\dots\wedge e_d$ is the canonical unit $d$-vector in $\rd$.
We apply now both sides of the above equality to the differential form $\varphi_0$. 
We get by the Stokes formula
$$\partial\langle T,f\circ\pi_0,r\rangle(\varphi_0)=
\langle T,f\circ\pi_0,r\rangle(d\varphi_0)=0,$$
since $d\varphi_0=(-1)^{d-1}\omega_d^{-1}(\pi_1)^{\#}\Omega_d$ and $(\pi_1)_{\#}T$ is supported in the unit sphere which has $d$-dimensional measure zero. Hence,
$$N(f,r)(\varphi_0)=\tilde{u}_{\#}\langle\Ha^d\wedge e_{1\dots d},f,r\rangle(\varphi_0).$$
If $0<f(x)=r<\ep$ and $Df(x)$ exists, let $\tau(x)$ be the unit $(d-1)$-vector associated with $\Tan(f^{-1}\{r\},x)$ and oriented so that $\langle\tau(x)\wedge\nabla f(x),\Omega_d\rangle>0$. 
Then, we can write 
$\langle\Ha^d\wedge e_{1\dots d},f,r\rangle=(\Ha^{d-1}\llc f^{-1}\{r\})\wedge \tau$ 
and obtain
\begin{eqnarray*}
\tilde{u}_{\#}\langle\Ha^d\wedge e_{1\dots d},f,r\rangle(\varphi_0)
&=&\int_{f^{-1}\{r\}}\langle\tau(x),\tilde{u}^{\#}\varphi_0(x)\rangle\, \Ha^{d-1}(dx)\\
&=&\int_{f^{-1}\{r\}}\langle(\bw_{d-1}D\tilde{u})\tau(x),\varphi_0(\tilde{u}(x))\rangle\, \Ha^{d-1}(dx)\\
&=&(d\omega_d)^{-1}\int_{f^{-1}\{r\}}\langle(\bw_{d-1}Du)\tau(x)\wedge u(x),\Omega_d\rangle\, \Ha^{d-1}(dx).
\end{eqnarray*}
Let $u_r:\{f=r\}\to S^{d-1}$ denote the restriction of $u$ to the oriented Lipschitz surface $\{f=r\}$ with values in the unit sphere.
We can write
$$\langle(\bw_{d-1}Du)\tau(x)\wedge u(x),\Omega_d\rangle=\det Du_r(x)$$
as the determinant of the differential $Du_r(x): \Tan(\{f=r\},x)\to\Tan(S^{d-1},u_r(x))$ with respect to any positively oriented orthonormal bases of the tangent spaces,
and applying the Area formula to the last integral, we obtain
\begin{eqnarray*}
N(f,r)(\varphi_0)&=&(d\omega_d)^{-1}\int_{S^{d-1}}\sum_{x\in u_r^{-1}\{v\}}
\frac{\det Du_r(x)}{|\det Du_r(x)|}\, \Ha^{d-1}(dv)\\
&=&(d\omega_d)^{-1}\int_{S^{d-1}}\deg(u_r,v)\, \Ha^{d-1}(dv)
\end{eqnarray*}
with the Brouwer degree of $u_r$, see \cite[\S VIII.4]{Dold80},
cf.\ \cite[p.~27]{Mil65} in the smooth case.
The degree $\deg(u_r,v)=\deg u_r$ is independent of $v\in S^{d-1}$ (see \cite[p.~28]{Mil65}) and by the Hopf theorem (see \cite[\S VIII.4.9]{Dold80}), $\deg u_r=\chi(\{f\leq r\})$. Thus we get
\begin{equation}  \label{degree}
N(f,r)(\varphi_0)=\deg(u_r)=\chi(\{f\leq r\})
\end{equation}
and, letting $r$ tend to $0$ and applying Corollary~\ref{Cor:propertiesOfRetracts}~(1), we complete the proof.
\end{proof}

\section{Singular sets of convex and DC functions}  \label{Sec_sing}
This section collects a few results on smallness of certain sets of singularities of convex and DC functions that will be needed in the sequel. We start with two propositions which are consequences of results from \cite{RZ}.

\begin{proposition}\label{dczlomy}   
Let $X$ be a $d$-dimensional Hilbert space, $1\leq k <d$, $\Omega \subset X$ an open convex set, $g$ and $h$ Lipschitz convex functions on $\Omega$, $\ep>0$ and denote $f\coloneqq g-h$.
For each $k$-dimensional subspace $K \subset X$ denote
\begin{eqnarray*}
Z_{\ep}^K&\coloneqq & \{x \in \Omega:\ f'_+(x,v)+ f'_+(x,-v) > \ep\text{ whenever }v\in K\text{ and }|v|=1\},\\
Z_{\ep}&\coloneqq & \bigcup \{ Z_{\ep}^K:\, K\text{ is a } k-\text{dimensional subspace of }X\}.
\end{eqnarray*}
Then 
\begin{enumerate}
\item[{\rm (i)}] each  $Z_{\ep}^K$  can be covered by finitely many  $(d-k)$-dimensional  DC surfaces associated with $K$,
\item[{\rm (ii)}] $Z_{\ep}$ can be covered by finitely many $(d-k)$-dimensional DC surfaces.
\end{enumerate}
\end{proposition}

\begin{proof}
The case $h=0$ (hence, $f$ is Lipschitz convex) was shown in \cite[Lemma~4.3, Proposition~4.4]{RZ}. 
In the general case, consider the Lipschitz convex function $\sigma\coloneqq  g+h$ and note that for each $x \in Z_{\ep}^K$ 
 and each unit vector $v \in K$, we have
\begin{eqnarray*}
\ep &<& |g'_+(x,v) - h'_+(x,v) + g'_+(x,-v) - h'_+(x,-v)|  \\
&\leq&  (g'_+(x,v) + g'_+(x,-v)) + ( h'_+(x,v) + h'_+(x,-v)) =  \sigma'_+(x,v) + \sigma'_+(x,-v).
\end{eqnarray*}
Hence, the assertions follow from the first (convex) case.
\end{proof}

\begin{proposition}[{\cite[Corollary 4.5]{RZ}}]\label{dim1}
Let $\Omega \subset \R^d$ be an open convex set,  $f$ a Lipschitz convex function on $\Omega$,  and $\ep>0$. Then the set
$$ \{x \in \Omega: \ \diam(\partial f(x)) > \ep\}$$
 can be covered by finitely many  DC hypersurfaces.
\end{proposition}

As a corollary we obtain the following result.

\begin{lemma}\label{dcsubd1}
Let $\Omega \subset \R^d$ be an open convex set, $\ep>0$ and $f = g-h$, where $g$ and $h$ are Lipschitz convex functions on $\Omega$. 
 Then the set
$$M\coloneqq  \{ x \in \Omega:\ f(x)=0\text{ and there exists }\alpha \in \partial f(x)\text{ with }|\alpha| > \ep\}$$
 can be covered by finitely many  DC hypersurfaces.
\end{lemma}
\begin{proof}
Let $x \in M$.
Since  $\partial f(x) \subset \partial g(x) - \partial h(x)$ (see  \cite[Proposition 2.3.3]{Cl90}), there exist $\alpha_1 \in \partial g(x) $ and   $\alpha_2 \in \partial h(x) $  such that $|\alpha_1 - \alpha_2| > \ep$.
 Set  $\mu(y)\coloneqq   \max(g(y), h(y))$, $y \in \Omega$. Then $\mu$ is clearly Lipschitz and convex.
 Since $g(x)= h(x)$, we have $\alpha_1 \in \partial \mu(x) $ and $\alpha_2 \in \partial \mu(x) $ (see \cite[Proposition~2.3.12]{Cl90}).
 So $\diam \partial \mu(x) > \ep$ and the assertion of the lemma follows from
 Proposition~\ref{dim1}.
\end{proof}

\begin{corollary}\label{dcsubd2}
Let  $\varnothing \neq C\subset \R^d$ be a bounded set  and
 $f$  a  DC function on $\R^d$ such that $f(x)=0$ for every $x \in C$. Let
 there exist
 $\ep>0$ such that for each $x \in C$ there exists $y^* \in \partial f(x)$ with
 $|y^*| > \ep$. Then $C$ can be covered by finitely many DC hypersurfaces.
\end{corollary}

\begin{proof}
Write  $f=g-h$ with convex $g$, $h$ and choose an open ball $B$ containing $C$. 
 Extend $g|_B$ and $h|_B$ to Lipschitz convex functions $h^*$ and $g^*$on $\R^d$, respectively. Now it is sufficient to apply Lemma \ref{dcsubd1} with $\Omega= \R^d$ and
 $f^*= g^* - h^*$.
 \end{proof}

\section{Results on the structure of WDC sets in $\R^d$}

\subsection{Boundaries of (locally) WDC sets}

\begin{proposition}\label{wdcpokr}
For each compact 
WDC set $M \subset \R^d$, its boundary $\partial M$ can be covered by finitely many  DC hypersurfaces.
\end{proposition}
\begin{proof}
If $M\neq\varnothing$, let $f=f_M$ and $\ep=\ep_M$ be as in Notation \ref{fixaury}. Using  Lemma \ref{DL}~(i) and \eqref{uzcl}
 we obtain that for each $x \in K\coloneqq  \partial M$ there exists $y^* \in \partial f(x)$ with
 $|y^*| \geq \ep$. Consequently Corollary \ref{dcsubd2} implies our assertion.
\end{proof}
Using Proposition \ref{P_local}, we easily obtain the following corollary.
\begin{corollary}\label{locpokr}
For each closed locally 
WDC set $M \subset \R^d$, its boundary $\partial M$ can be locally covered by finitely many  DC hypersurfaces.
\end{corollary}

We will  use the following terminology repeatedly.

\begin{definition}\label{ak}
 Let  $0\leq k \leq d$ and $A \subset \R^d$. Then we denote by $A^{[k]}$ the set of all
 $x \in A$, at which $A$ locally coincides with a $k$-dimensional DC surface. 
\end{definition}
Obviously, $A^{[k]}$ is open in $A$, $A^{[d]}= \INt A$ and $A^{[0]}$ is the set of all isolated points of $A$. Further, if $0\leq k\leq d$ and   $A^{[k]} \neq \varnothing$, then 
 $A^{[k]}$ is a DC manifold of dimension $k$.

\begin{proposition}\label{hranice}
Let $M \subset \R^d$ be a closed locally WDC set such that  $M = \overline{\INt M}$. Then
 $(\partial M)^{[d-1]}$ is dense in $\partial M$.
\end{proposition}
\begin{proof}
Choose an arbitrary $b \in \partial M$ and $r>0$. It is sufficient to prove that $B(b, r) \cap
 (\partial M)^{[d-1]} \neq \varnothing$. By Corollary \ref{locpokr} we can choose a relatively
 open subset $H$ of $\partial M$ with $b \in H \subset B(b,r)$ 
 and $(d-1)$-dimensional DC surfaces $P_1,\dots,P_s$
 which cover $H$. Since $H$ is  locally compact, it is a Baire space (i.e., the Baire theorem holds in $H$), see
 \cite[pp.~249, 250]{Dug}. Consequently there exists an index $j$ such that $P_j \cap H$ is not nowhere dense in $H$. Since $P_j$ is closed, there exist $z \in H$ and $\delta>0$ such that
 $B(z,\delta) \cap \partial M \subset P_j$. Choose a $(d-1)$-dimensional space
 $W\subset \R^d$, a unit vector $v \in W^{\perp}$ and a DC function $\vf: W \to \R$
 such that  $P_j = \{t+ \vf(t) v:\ t \in W\}$. Using continuity of $\vf$, we can clearly
 choose a ball $B(t_0, \omega)$ in $W$ and $\eta>0$ such that $z= t_0 + \vf(t_0) v$ and
$$ \{t + (\vf(t) + \tau) v:\  t \in B(t_0, \omega), \tau \in [-\eta, \eta]\} \subset B(z, \delta).$$
Then the sets
$$A_1\coloneqq   \{t + (\vf(t) + \tau) v:\  t \in B(t_0, \omega), \tau \in (0, \eta)\}\ \ \text{and}$$
$$A_2\coloneqq   \{t + (\vf(t) + \tau) v:\  t \in B(t_0, \omega), \tau \in (-\eta, 0)\}$$
 are clearly open connected sets which do not intersect  $\partial M$ and so,
 for  $i \in \{1,2\}$, either $A_i \cap M = \varnothing$ or $A_i \subset M$. The case $A_1 \subset M$, $A_2 \subset M$ is impossible, since it clearly implies $z \in \INt M$. Also the case $A_1 \cap M = \varnothing$, $A_2 \cap M = \varnothing$ is impossible, since it implies $ z \notin  \overline{\INt M}$.
 So either $A_1 \subset M$  and  $A_2 \cap M = \varnothing$ or $A_2 \subset M$  and  
$A_1 \cap M = \varnothing$. In both cases we clearly have $z \in \partial M^{[d-1]} \cap B(b,r)$ and 
our assertion follows.
\end{proof}

\begin{remark}\label{velmir}
\cite[Example 7.12(i)]{RZ} shows that, in the above proposition, we cannot assert that
 $\partial M \setminus (\partial M)^{[d-1]}$ has zero $(d-1)$-dimensional Hausdorff measure
 even in the case when $M\subset \R^2$ is a set of positive reach.
\end{remark}

\subsection{Lipschitz manifolds and sets with Lipschitz boundaries}

In Section \ref{plane} we will completely characterize locally WDC sets in $\R^2$. For sets in $\R^d, \ d\geq 3,$ we do not know such a characterization (even for sets with locally positive reach). 

However, we will give such full characterization for sets of ``special types'', namely for
 $k$-dimensional Lipschitz manifolds (cf.\ Definition~\ref{plochy}~(b)) and for ``closed Lipschitz domains''.

The easier implication of the following result
 was observed in \cite[Proposition~3.2~(ii)]{PR13}; we present a full proof for completeness.

\begin{proposition}\label{lipman}
Let a closed set $A \subset \R^d$ be a  Lipschitz manifold of dimension $0<k<d$. Then $A$ is a locally WDC
 set if and only if $A$ is a DC manifold of dimension $k$.
\end{proposition}
\begin{proof}
First suppose that $A$ is a DC manifold of dimension $k$ and $a\in A$. 
 Using Lemma \ref{vldc}~(viii), it is easy to see that 
 there exist $W \in G(d,k)$, a DC mapping $\varphi:W\to W^\perp$ and $r>0$ such that
$$A\cap B(a,r)=\{ y+\varphi(y):\, y\in W\} \cap B(a,r).$$
Consider the function
$$f: x\mapsto | x-\phi(\pi_W(x))|,\quad x\in \R^d,$$
where $\phi(y)\coloneqq y+\varphi(y)$, $y\in W$. By Lemma \ref{vldc}~(i), (iii), (iv), (vi),  $f$ is a DC function and for each $x$ with $f(x)>0$ we have
$$f'(x,\nu(x))=1\quad\text{at}\quad \nu(x)\coloneqq \frac{x-\phi(\pi_W(x))}{| x-\phi(\pi_W(x))|}.$$
It follows that $f'(y)\cdot\nu(y)\geq 1$ whenever $f'(y)$ exists and, since $\nu$ is continuous on $\{f>0\}$, also $v\cdot\nu(x)\geq 1$ whenever $f(x)>0$ and $v\in\partial f(x)$. Thus, $0$ is a weakly regular value of $f$ and $f$ is a DC aura for $\{f=0\}$. Since $A\cap B(a,r)=\{f=0\}\cap B(0,r)$ by construction, we conclude that  $A$ is a locally WDC set.

To prove the opposite implication, 
suppose that $A$ is a locally WDC set.
Consider an arbitrary $a\in A$.
Using Definition \ref{lwdc} and Proposition \ref{P_local}, we can choose $\delta>0$ and a compact
 WDC set $M\subset A$ such that $M \cap B(a,\delta) = A \cap B(a,\delta)$. Let $f=f_M$ and
 $\ep= \ep_M$ be as in Notation \ref{fixaury}. Diminishing $\delta$, if necessary, we can suppose
that
 there exist $W \in G(d,k)$,
 a Lipschitz mapping  $\vf: W \to W^{\perp}$, $\delta>0$ and an open set $U$ in $W$ such
 that
$$  B(a,\delta) \cap A = B(a,\delta) \cap M = \{w + \vf(w):\ w \in U\}.$$
Fix $L>1$ such that $\vf$ is $L$-Lipschitz.
 Now consider  arbitrary  $v\in
 W^{\perp} \cap S^{d-1}$, $x \in B(a,\delta) \cap M$, $t>0$
 and $w\in U$. We will observe that, denoting  $w_0 = \pi_W(x)$, we have
\begin{equation}\label{odvz}
 d\coloneqq  |(x+tv)-(w+ \vf(w))| = |(w_0 + \vf(w_0)+tv)-(w+ \vf(w))| \geq   \frac{t}{2L}. 
\end{equation}
Indeed, the inequality is obvious, if $|w-w_0| \geq   \frac{t}{2L}$. And in the opposite case
 we have
$$d \geq  |\vf(w_0)+tv - \vf(w)| \geq t- |\vf(w_0) - \vf(w)| \geq t - L|w-w_0| \geq \frac{t}{2}
\geq   \frac{t}{2L}.$$ 
Consequently there exists
 $t_0>0$ such that $\dist(x+tv, M) \geq   \frac{t}{2L}$ for each
 $t \in (0, t_0)$.
Using also Lemma \ref{DL}~(v), we obtain that $f(x+tv) \geq \frac{\ep}{2}  \frac{t}{2L}$
 for $0 < t < \min(t_0,\ep)$.
Since $f(x)=0$, we have proved that
  for each  $x \in B(a,\delta) \cap M$ and  $ v\in W^{\perp}\cap S^{d-1}$ we have $f'_+(x,v) \geq  \frac{\ep}{4L}$. So Proposition \ref{dczlomy}~(i) (applied with $\Omega\coloneqq  B(a, \delta)$ and 
	 $K\coloneqq   W^{\perp}$; note that if $f=g-h$, where $g$, $h$ are convex on $\R^d$, then $g$, $h$ are
	 necessarily Lipschitz on each bounded set)
	implies that
 $B(a,\delta) \cap M$ can be covered by finitely many DC surfaces associated with $W^{\perp}$.
 Using Lemma \ref{vldc}~(ix), we easily obtain that $\vf$ is locally DC on $U$.
Therefore,  $A$ is a $k$-dimensional DC manifold.
\end{proof}

\begin{remark}\label{topvar} We do not know whether each $k$-dimensional topological manifold which is
 locally WDC is a DC manifold. However, it is a DC manifold except a nowhere dense set, see Corollary
 \ref{topman}.
\end{remark}

\begin{definition}\label{dom}
We will say that a closed set $A \subset \R^d$ is a {\it closed Lipschitz domain} (resp. a
 {\it closed DC domain}) if for each $a \in \partial A$ there exist $r>0$,
 $W \in G(d,d-1)$,
 $v \in W^{\perp} \cap S^{d-1}$ and a Lipschitz (resp. DC) function $\vf: W \to \R$
 such that 
\begin{equation}\label{domhr}
A \cap B(a,r) = \{w+tv:\ w \in W, t \leq \vf(w)\} \cap B(a,r).
\end{equation}
\end{definition}
\begin{remark}\label{dcdom}
Closed DC domains are called simply ``DC domains'' in \cite{PR13} and are considered as most natural 
 examples of closed locally WDC sets (see \cite[Proposition 3.3~(i)]{PR13} and the title of \cite{PR13}).
\end{remark}
\begin{theorem}\label{wdcdom}
Let $A \subset \R^d$ be a closed Lipschitz domain. Then $A$ is locally WDC if and only if 
 $A$ is a closed DC domain.
\end{theorem}

\begin{proof} 
Each closed DC domain is locally WDC by \cite[Proposition 3.3~(i)]{PR13}.
The following proof of the converse implication is quite analogous to the corresponding part of the proof of Proposition \ref{lipman}.

Suppose that $A\neq\varnothing$ is locally WDC and consider a point $a\in \partial A$. Choose $r$, $W$, $v$ and a Lipschitz $\vf: W \to \R$
 as in Definition \ref{dom}; so  \eqref{domhr} holds.
Using Definition \ref{lwdc} and Proposition \ref{P_local}, we can choose $0<\delta<r$ and a compact
 WDC set $M\subset A$ such that $M \cap B(a,\delta) = A \cap B(a,\delta)$. Let $f=f_M$ and
 $\ep= \ep_M$ be as in Notation \ref{fixaury}.
 It is easy to see that 
\begin{equation}\label{stopy}
\partial M \cap B(a,\delta) = 
\{w+\vf(w)v:\ w \in W\} \cap B(a,\delta)= \{w+\vf(w)v:\ w \in U\}
\end{equation}
for an open subset $U$ of $W$. Fix $L>1$ such
 that $\vf$ is $L$-Lipschitz.  Now consider  arbitrary $x \in \partial A \cap B(a,\delta)$,
 $t>0$ and $w\in U$. The same argument which was used to prove \eqref{odvz} reveals that
 $$ |(x+tv)-(w+ \vf(w) v)| \geq \frac{t}{2L}.$$ 
This inequality together with \eqref{domhr} and \eqref{stopy} easily implies  that
  there exists
 $t_0>0$ such that $\dist(x+tv, M) \geq   \frac{t}{2L}$ for each
 $t \in (0, t_0)$.
Using also Lemma \ref{DL}~(v), we obtain that
 $f(x+tv) \geq \frac{\ep}{2}  \frac{t}{2L}$
 if $0 < t < \min(t_0,\ep)$.
Since $f(x)=0$, we have proved that
  for each  $x \in B(a,\delta) \cap \partial M$    we have $f'_+(x,v) \geq  \frac{\ep}{4L}$
	 (and  $f'_+(x,-v) \geq 0$, since $f$ is nonnegative).
	So Proposition \ref{dczlomy}~(i) 
implies (cf. the corresponding argument in the  proof of Theorem \ref{lipman}) that
 $B(a,\delta) \cap M$ can be covered by finitely many DC surfaces associated with $W^{\perp}$.
 Using Lemma \ref{vldc}~(ix), we easily 
obtain that $\vf$ is locally DC on $U$.
Therefore, using  Lemma \ref{vldc}~(viii), we easily obtain that  $A$ is a closed DC domain.
\end{proof} 
\begin{remark}
Note that in Theorem~\ref{wdcdom} it is not enough to assume that $A$ is merely a closed ``topological domain''; consider the set $A=\{(x,y)\in\R^2:\, |y|\leq x^2, x\geq 0\}$ which is even a set with positive reach, but not a closed Lipschitz domain.
\end{remark}

\subsection{A technical lemma}

 Recall that by $G(d,k)$ we denote the set of all $k$-dimensional linear subspaces of $\R^d$.  
 Now we will define two (well-known) notions which substitute the notions of the angle between a vector and a subspace of $\R^d$ and the angle between two elements of  $G(d,k)$, but are more suitable for 
our purposes.

\begin{definition}\label{uhly}
Let $0 \neq v \in \R^d$,  $0<k<d$   and $V, W  \in G(d,k)$. Then we define
$$ \sigma(v, W) \coloneqq  \frac{\dist(v,W)}{|v|},$$
$$ \gamma(V,W)\coloneqq  \max \left( \sup \{\dist(v, W):\ v \in V \cap S^{d-1}\}, \sup \{\dist(w, V):\ w \in W \cap S^{d-1}\}\right).$$
\end{definition}

\begin{lemma}\label{siga}
Let $0 \neq z \in \R^d$,  $0<k<d$   and $V, W  \in G(d,k)$.        
Then  $\sigma(z,W) \leq \sigma(z,V) + \gamma(V,W)$. 
\end{lemma}
\begin{proof}
We can and will suppose that $|z|=1$.
Note that for $v\coloneqq  \pi_V(z)$ we have
 $|z-v| = \sigma(z,V)$.
 Since $|v|\leq 1$, we have $\dist(v,W) \leq \gamma(V,W)$. Consequently
$$\sigma(z,W) = \dist(z,W) \leq |z-v| + \dist(v,W) \leq \sigma(z,V) + \gamma(V,W).$$
\end{proof}

\begin{lemma}\label{VWZ}
Let  $0<k<d$ and  $V, W  \in G(d,k)$. Let $\gamma\coloneqq  \gamma(V,W) >0$.
Then there exists a subspace $Z$ of $\R^d$ of dimension $d-k+1$ such that
\begin{equation}\label{bune}
\text{for each $0 \neq z \in Z$, either $\sigma(z,V)\geq \gamma/3$ or  $\sigma(z,W)\geq \gamma/3$.}
\end{equation}
\end{lemma}
\begin{proof}
 Without a loss of generality we can suppose that there exists
 $v \in V \cap S^{d-1}$ with $\dist(v,W)= \gamma$. Set $Z\coloneqq  \spa (\{v\} \cup  V^{\perp})$. 
 Clearly $\dim Z = d-k+1$.
To prove \eqref{bune}, consider an arbitrary $z \in Z  \cap S^{d-1}$.
 Write $z= \lambda v + u$, where $u \in  V^{\perp}$. Now distinguish two possibilities.
If $|u| \geq \gamma/3$, then $\dist(z,V)= \dist(u,V) = |u| \geq \gamma/3$.
If $|u| < \gamma/3$, then $|\lambda| \geq 2/3$, since $\gamma \leq 1$ and 
$|\lambda| + |u| \geq |z|=1$. Consequently  $\dist (\lambda v, W) \geq (2/3) \gamma$, and so
 $\dist (z, W) \geq \dist (\lambda v, W) - |u| \geq  (2/3) \gamma - \gamma/3 = \gamma/3$.
\end{proof}

\begin{lemma}\label{silip}
Let $0<k<d$ and  $W \in G(d,k)$. Let $\varnothing \neq S \subset \R^d$, $1/2 > \ep>0$ and
\begin{equation}\label{mauh}
\sigma(s_2-s_1, W) \leq  \ep\ \ \ \text{whenever}\ \ \ s_1,s_2 \in S, s_1 \neq s_2.
\end{equation}
Then there exists a $(2\ep)$-Lipschitz mapping $g: \pi_W(S) \to W^{\perp}$ such that
 $S= \{w+ g(w):\ w \in \pi_W(S)\}$. 
 \end{lemma}
\begin{proof}
Let $s_1,s_2 \in S, s_1 \neq s_2$ and $s_i= w_i + z_i$, where $w_i \in W$ and $z_i \in W^{\perp}$,
 $i=1,2$. Set $s\coloneqq  s_2-s_1$, $w\coloneqq  w_2 - w_1$, $z\coloneqq z_2-z_1$. Our task is to prove  $|z| \leq  2 \ep |w|$.
   Clearly $s= w +z$ and by \eqref{mauh} $\sigma(s,W) = \frac{|z|}{|s|} \leq \ep$.
  Therefore
	 $|w| \geq |s|- |z| \geq (1-\ep) |s| \geq (1/2) |s|$ which implies  $|z|\leq \ep |s| \leq    2 \ep |w|$.
	\end{proof}

\begin{lemma}\label{transv}
	 Let $P\coloneqq  P_1\cup\dots \cup P_s$, where all $P_i \subset \R^d$
 are DC surfaces of dimension $k$, $1\leq k < d$. Let $M \subset \R^d$ be a compact set and $\omega >0$. Then there exists a set $E \subset \R^d$ which is a finite
 union of DC surfaces of dimension $k-1$ such that for each $x \in (P\cap M) \setminus E$ there exist $\eta(x)>0$ and  $W \in G(d,k)$ such that, denoting $U\coloneqq  B(x,\eta(x))$,  the following conditions hold:
  \begin{enumerate}
  \item[(i)]  
  If  $x \notin P_i$, then $P_i \cap  U = \varnothing$.
  \item[(ii)]
  If $x \in P_i$, then there exists an $\omega$-Lipschitz, locally DC mapping $\psi_i: D_i \to W^{\perp}$ such that $D_i \subset W$ is open in $W$ and
   $P_i  \cap U=    \{t+ \psi_i(t):\ t \in D_i\}$.
   \end{enumerate}
  \end{lemma}
	\begin{proof}
	Fix an open ball $\Omega$ containing $M$ and set $\alpha\coloneqq  \omega/(40 d)$.
	For each $1 \leq i \leq s$ choose  $W_i \in G(d,k)$ and a Lipschitz DC mapping
	 $\vf^i: W_i \to W_i^{\perp}$ such that $P_i = \{t + \vf_i(t): \ t \in W_i\}$.
	Set $\pi_i\coloneqq  \pi_{W_i}$  and  $\Omega_i\coloneqq  \pi_i(\Omega)$.
	\smallskip
	For each $i$, fix an orthonormal basis $(b_1^i,\dots,b_{d-k}^i)$ of   $W_i^{\perp}$.  Then
	 there exist Lipschitz  DC functions $\vf_1^i,\dots, \vf_{d-k}^i$ on $W_i$ such that
	\begin{equation}\label{devf}
	 \vf^i(t)= \vf_1^i(t) b_1^i + \dots +   \vf_{d-k}^i(t) b_{d-k}^i,\ \ \ t\in W_i.
	\end{equation}
	Write $\vf_j^i = g_j^i - h_j^i$, where $g_j^i$  and  $h_j^i$ are convex functions on $W_i$.
	Applying Proposition~\ref{dim1} to  
	 Lipschitz convex functions  $g_j^i|_{\Omega_i}$ and $h_j^i|_{\Omega_i}$
	we easily
	 obtain in $W_i$ a
	 set $N_i$ which is a finite union of DC surfaces of dimension $k-1$ such that, for each $1\leq j \leq d-k$, 
	$$  \diam   \partial g_j^i(t) < \frac{\alpha}{2d}\ \ \ \text{and}\ \ \   \diam   \partial h_j^i(t) < \frac{\alpha}{2d}\ \ \ \text{for each}\ \ \  t \in \Omega_i \setminus N_i.               $$
	Since  $\partial \vf_j^i (t) \subset \partial g_j^i(t) - \partial h_j^i(t)$ 
	(see  \cite[Proposition 2.3.3]{Cl90}), we obtain
	(for $i$, $j$ and $t$ as above) that $\diam \partial \vf_j^i (t) < \alpha/d$. For those $i$, $j$ and $t$, choose  $L_j^{t,i} \in \partial \vf_j^i (t)$. 
	 Using the upper semicontinuity of the Clarke subdifferential (see \eqref{uzcl}), we can find
	 for each $t \in \Omega_i \setminus N_i$ its neighbourhood $U_i^t$  in $W_i$ such that, for each
	 $j$,
	\begin{equation}\label{shpo}
	|L- L_j^{t,i}| < \frac{\alpha}{d}\ \ \ \text{whenever}\ \ \ L \in \partial \vf_j^i (\tau)\ \ \ \text{for some}\ \ \ \tau \in U_i^t.
	\end{equation}
	Define the linear mapping $L^{t,i}: W_i \to W_i^{\perp}$  by the formula
	\begin{equation}\label{delit}
	L^{t,i}(u)=L_1^{t,i} (u) b_1^i + \dots +  L_{d-k}^{t,i} (u)  b_{d-k}^i,\ \ \ u\in W_i.
	\end{equation}
	We will show that, for each $t \in  \Omega_i \setminus N_i$ and each $j$,
	\begin{equation}\label{eli}
	|\vf^i(t_2) - \vf^i(t_1) - L^{t,i}(t_2-t_1)| < \alpha |t_2-t_1|\ \ \text{whenever}\ \ t_1, t_2 \in
	U_i^t.
	\end{equation}
	To this end choose arbitrary $t_1, t_2 \in	U_i^t$. By Lebourg's mean value theorem
	 (\cite[Theorem 2.3.7]{Cl90}) we have, for each $j$, $\vf_j^i (t_2) -  \vf_j^i (t_1) = 
	 L(t_2-t_1)$, where $L \in \partial \vf_j^i (\tau)$ for some  $\tau \in U_i^t$. Therefore,
	 using also \eqref{shpo}, \eqref{delit} and  \eqref{devf}, we easily obtain \eqref{eli}.
	
	Set $E_1^i\coloneqq   \{t+ \vf^i(t):\ t \in N_i\}$ and $E_1\coloneqq   \bigcup_{i=1}^sE_1^i$. Then $E_1$ is  a finite
 union of DC surfaces of dimension $k-1$ by Lemma \ref{sloz}. Further, for $x \in (P_i \cap M)\setminus E_1$,
	set $V_i^x\coloneqq  \{u+ L^{\pi_i(x),i}(u):\ u \in W_i\}$. For each couple $1\leq i < j \leq s$ define
	$$B_{i,j} = \{x\in (P \cap M)\setminus E_1: 
	 \ x\in P_i \cap P_j\ \text{and}\ \ \gamma(V_i^x,V_j^x)> (3/8) \omega\},$$
	$$ E_2: = \bigcup_{1\leq i < j \leq s} B_{i,j}\ \ \  \text{and}\ \ \ E\coloneqq  E_1 \cup E_2.$$
	To finish the proof, it is sufficient to show that
	\begin{equation}\label{edva}
	\text{$E_2$ is contained in  a finite
 union of DC surfaces of dimension $k-1$}
\end{equation}
and that
\begin{equation}\label{etaw}
\text{for each $x \in (P\cap M) \setminus E$ there exist $\eta(x)$ and  $W$ 
 for which (i), (ii) hold.}
\end{equation}
 We will infer 	\eqref{edva} from \cite[Lemma~4.3]{RZ} (or Proposition~\ref{dczlomy}).
Indeed, we will construct a Lipschitz
 convex function $f$ on $\Omega$ and, for each couple $1\leq i < j \leq s$ and   $x\in B_{i,j}$, a $(d-k+1)$-dimensional space $V^x_{i,j} \subset \R^d$ such
 that 
\begin{equation}\label{cil}
\text{$f'_+(x,v) + f'_+(x,-v) > \alpha$ whenever  $v \in V^x_{i,j}$ with $|v|=1$.}
\end{equation}

First we will construct, for each $1\leq i \leq s$, a convex function $f_i$ on $\R^d$.
 To this end, fix $i$. For a while, we will write for the short
 $\vf, b_j, \vf_j, g_j, h_j$ instead of  $\vf^i, 
b_j^i, \vf_j^i, g_j^i, h_j^i$.
 For each $x \in \R^d$, we will define the ``new coordinates''  $t(x)\in W_i$ and $y_j(x) \in \R$
 ($j= 1,\dots,d-k$) of the point $x$ by the equality  $x= t(x)+y_1(x)b_1+\dots + y_{d-k}(x)b_{d-k}$.
 Now for $1 \leq j \leq d-k$
we define a function $s_j$ on $\R^d$ by the formula 
$$  s_j(x) = \max (y_j(x) - \vf_j(t(x)), 0),\ \ \ x \in \R^d.$$
Since each $s_j$ is by Lemma \ref{vldc}~(i), (iii), (iv), (vi) a DC  function, we can write   $s_j = p_j-q_j$, where $p_j$ and 
 $q_j$ are convex functions on $\R^d$. Set $f_i\coloneqq  \sum_{j=1}^{d-k} (p_j + q_j)$. 
 We will show that, for each $x\in (P_i \cap M) \setminus E_1$,
\begin{equation}\label{odchi}
(f_i)'(x,z) + (f_i)'(x,-z) >\alpha\ \ \ \text{whenever}\ \ \ |z|=1\ \ \ \text{and}\ \ \ \sigma(z,V^x_i)> \omega/8.
\end{equation}
To this end, let $x\in (P_i \cap M) \setminus E_1$,  $|z|=1$ and $\sigma(z,V^x_i)> \omega/8$.
Set $t\coloneqq  t(x)$ and $w\coloneqq  t(z)$. Set $L\coloneqq  L^{t,i}$; let $L(u)= L_1(u)b_1+\dots+L_{d-k}(u)b_{d-k}$, $u \in W_i$.
Since $\dist(z, V_i^x) > \omega/8$, we have 
$$|(w + y_1(z)b_1+\dots + y_{d-k}(z) b_{d-k})   - (w + L_1(w)b_1+\dots +L_{d-k}(w) b_{d-k}) | > \omega/8.$$
So we can find  $1 \leq j \leq d-k$ such that $|y_j(z) - L_j(w)| > \omega/(8d) = 5\alpha$.
First suppose that
\begin{equation}\label{yjv}
y_j(z) \geq L_j(w) + 5 \alpha. 
\end{equation}

By \eqref{eli} we easily obtain, for all sufficiently small $h>0$, the inequality
$$ |\vf_j(t+hw) - \vf_j(t) - L_j(hw)| < \alpha h |w|.$$
So, since $\vf_j(t)=y_j(x)$, using \eqref{yjv}, we obtain for these $h$
\begin{equation}\label{odh1}
 \vf_j(t+hw) < y_j(x) + h L_j(w) + \alpha h |w|\leq  y_j(x) + h (y_j(z) - 5 \alpha) + \alpha h =
  y_j(x) + hy_j(z)-4 \alpha h.
\end{equation}
Therefore $s_j(x+hz) = \max (y_j(x) +h y_j(z) - \vf_j(t+hw), 0) \geq 2 \alpha h$.
 Since $s_j(x)= \max(y_j(x) - \vf_j(t),0)=0$, we
 obtain $(s_j)'_+(x,z) \geq 2 \alpha$.
Since $s_j \geq 0$ and $s_j(x)=0$, we have
  $(s_j)'_+(x,-z) \geq 0$, and thus $(s_j)'_+(x,z)+ (s_j)'_+(x,-z) \geq 
	2 \alpha$. Therefore, it is easy to see that  
	\begin{equation}\label{fizl}
	(f_i)'_+(x,z) + (f_i)'_+(x,-z) >  \alpha.
	\end{equation}
	 The case $y_j(z) \leq  L_j(w) - 5 \alpha$ can be treated quite symmetrically. Indeed, in this
	 case \eqref{yjv} holds, if we write in it $-z$ and $-w$ instead of $z$ and $w$. Since
	 also all subsequent formulas till \eqref{fizl} hold after these substitutions, we obtain
	 that \eqref{fizl} holds also in the second case.

Now put $f\coloneqq  \sum_{i=1}^s f_i$.
Let $1\leq i < j \leq s$ and   $x\in B_{i,j}$ be fixed. By Lemma \ref{VWZ}, there exists a
 $(d-k+1)$-dimensional space $Z\eqqcolon V^x_{i,j} \subset \R^d$ such
 that 
\begin{equation}\label{bune2}
\text{for each $0 \neq z \in Z$, either $\sigma(z,V_i^x)> \omega/8$ or  $\sigma(z,V_j^x)> \omega/8 $.}
\end{equation}
By \eqref{odchi}, we obtain  \eqref{cil} in both possible cases.
 So, each $B_{i,j}$, and thus also $E_2$,  is contained in  a finite
 union of DC surfaces of dimension $k-1$
 by  \cite[Lemma~4.3]{RZ} (or Proposition~\ref{dczlomy}). Thus we have proved $\eqref{edva}$.

To prove  \eqref{etaw}, consider an arbitrary $x \in (P\cap M) \setminus E$.
 Choose a $1\leq j \leq s$ for which $x \in P_j$ and set $W\coloneqq  V_j^x$. 
Now consider an arbitrary $1\leq i \leq s$ for which $x \in P_i$. Set $t\coloneqq  \pi_i(x)$,
 $S_i\coloneqq  \{\tau+ \vf^i(\tau):\ \tau \in U_i^t\}$ and consider arbitrary different $s_1, s_2 \in S_i$.
 Then $s_1 = t_1 + \vf^i(t_1)$ and $s_2 = t_2 + \vf^i(t_2)$ for some $t_1, t_2 \in
	U_i^t$ and
	$ s_2-s_1 = (t_2 - t_1) + (\vf^i(t_2)- \vf^i(t_1))$. Set $u\coloneqq  (t_2-t_1) +  L^{t,i}(t_2-t_1)$.
	 Then  $u \in V_i^x$ and  \eqref{eli} implies
	$$ |(s_2-s_1) - u|= |\vf^i(t_2) - \vf^i(t_1) - L^{t,i}(t_2-t_1)| \leq  \alpha |t_2-t_1| \leq
	 \alpha |s_2-s_1|.$$
	Consequently $\sigma(s_2-s_1,V_i^x) \leq \alpha$. Since $x \notin E_2$ and $B_{i,j} \subset E_2$,
	we have  $\gamma(V_i^x, W) = \gamma(V_i^x,V_j^x ) \leq 3 \omega/8$. Using Lemma \ref{siga},
	 we obtain that
	\begin{equation}\label{mauh2}
\sigma(s_2-s_1, W) \leq  \alpha + 3 \omega/8 < \omega/2 \ \ \ \text{whenever}\ \ \ s_1,s_2 \in S_i, s_1 \neq s_2.
\end{equation}
Denote  $\pi\coloneqq  \pi_W$. By Lemma \ref{silip}, \eqref{mauh2}
 implies that
 there exists an $\omega$-Lipschitz mapping $g_i: \pi(S_i) \to W^{\perp}$ such that
 $S_i = \{\tau+ \vf^i(\tau):\ \tau \in U_i^t\} = \{w+ g_i(w):\ w \in \pi(S_i)\}$. 
  Setting  $\xi(t)\coloneqq  t+ \vf^i(t)$, $t \in U_i^t$, and $\kappa\coloneqq  \pi \circ \xi$, we
	 easily see that $\kappa$ is a homeomorphism of $U_i^t$ onto $\pi(S_i)$. Since both $W_i$ and $W=W_j$
	 are homeomorphic to $\R^k$, Brouwer's Invariance of Domain Theorem (see e.g.\ \cite[Ch.~IV, 7.4]{Dold80}) 
	implies that the set $\pi(S_i)$
	 is open in $W$. Using Lemma \ref{L:twoDCgraphs} we obtain that  $g_i$ is locally DC
	 on $\pi(S_i)$. Now choose $\eta(x)$ so small that (i) holds and $P_i \cap B(x, \eta(x)) \subset S_i$ for each $i$ with $x \in P_i$. Then (ii) clearly holds with $D_i\coloneqq  \pi(S_i)$ and $\psi_i\coloneqq g_i|_{D_i}$.
	\end{proof}

\subsection{Main results}

\begin{lemma}\label{Z}
	Let $f$ be a  DC function in $\R^d$, $c>0$, and let $P_1,\dots, P_s$ be 
	 DC surfaces of dimension $0 < k< d$ in $\R^d$. Let $A \subset P\coloneqq  P_1\cup\dots \cup P_s$ be a bounded set such that $f(x)=0$ for each $x \in A$.   Then there exists a set $T \subset \R^d$ which is a finite
 union of DC surfaces of dimension $k-1$ such that, if $x \in A \setminus T$, then
\begin{equation}\label{impl}
  \limsup_{p \to x, p \in P}  \frac{f(p)}{|p-x|} \leq c.
\end{equation}
\end{lemma}
	\begin{proof}
	Denote
	 $$  Z_i\coloneqq  \{x\in A \cap P_i:\   \limsup_{p \to x, p \in P_i}  \frac{f(p)}{|p-x|} > c\}.$$
	Since  \eqref{impl}  clearly holds for each  $x \in A \setminus \bigcup_{i=1}^s Z_i$, it is sufficient to prove that, for each fixed $1\leq i \leq s$,
	\begin{equation}\label{proi}
	\text{$Z_i$ can be covered by finitely many $(k-1)$-dimensional DC surfaces.}
	\end{equation}
  Since $P_i$ is a $k$-dimensional  DC surface, we can find $W \in G(d,k)$
	and a DC mapping $\vf: W \to W^{\perp}$ such that $P_i= \{w+\vf(w):\ w \in W\}$.
	 Set $\pi\coloneqq  \pi_W$ and $\Phi(w)\coloneqq  w + \vf(w),\ w \in W$.
	 Then $\Phi$ is a DC mapping and
  $\Phi\coloneqq  (\pi |_{P_i} )^{-1}$.
Set $D\coloneqq  \pi(Z_i)$ and 
  $\hat f\coloneqq  f \circ \Phi$. By Lemma \ref{vldc}~(i), (vi), (iv),  $\hat f$ is  a DC function
 on $W$ and $\hat f(y)=0$ for each $y \in D$. Now consider an arbitrary $d \in D$.
 Then $x\coloneqq  \Phi(d) \in Z_i$ and so there exists a sequence $(x_n)$ such that 
$x_n\in P_i$,\ $x_n \to x$ and  $\frac{f(x_n)}{|x_n-x|} \geq c$. Set $\hat x_n: = \pi(x_n)$.
  Since $\pi$ is $1$-Lipschitz,
 we have   $\hat f(\hat x_n) \geq  c |\hat x_n- d|$. 
  Lebourg's mean-value theorem (\cite[Theorem 2.3.7]{Cl90}) implies that there exist points
	 $u_n \in W$ and $\alpha_n \in \partial \hat f(u_n)$ such that $u_n \to d$ and
		$$ \langle \alpha_n, \hat x_n - d\rangle = \hat f(\hat x_n) - \hat f(d) = \hat f(\hat x_n) \geq  c |\hat x_n- d|.$$
 Consequently $|\alpha_n| \geq c$ and therefore   \eqref{uzcl} easily implies that
 there exists
 $y^* \in \partial \hat f(d)$ with $|y^*| \geq c$.
Since $D=\pi(Z_i)$ is bounded, identifying $W$ with $\R^k$ and
 using Corollary \ref{dcsubd2}, we obtain that $D$ can be covered  by finitely many
  $(k-1)$-dimensional DC surfaces  $S_1$,\dots, $S_m$ in $W = \R^k$. Then $Z_i$ is covered
 by $Q_i\coloneqq  \Phi(S_i)$, $i=1,\dots,m$, and each $Q_i$ is by Lemma \ref{sloz} a $(k-1)$-dimensional DC surface
 in $\R^d$; thus \eqref{proi} holds.
	\end{proof}
	
	\begin{lemma}\label{71h}
	Let  $0<k<d$ and $W \in G(d,k)$. Let  $D \subset W$ be a closed set, $\omega>0$
	 and let $\psi_i: D \to W^{\perp}$, $i=1,\dots,p$,  be $\omega$-Lipschitz mappings. 
	 Denote $H_i\coloneqq  \{(t,\psi_i(t)):\ t \in D\}$, $i=1,\dots,p$. Let $d \in D$, $z_0,\, z_1 \in   W^{\perp}$, and $\kappa: [0,1] \to \R^d$ be a continous mapping such that $\kappa(0)= d + z_0$,
	 $\kappa(1)= d +  z_1$ and $\kappa([0,1]) \subset H_1\cup\dots\cup H_p$. Then
	\begin{equation}\label{odhdiam}
	\diam \kappa([0,1]) \geq  \frac{|z_1-z_0|}{p \omega}.
	\end{equation}
	\end{lemma}
	\begin{proof} 
	Denote $\pi\coloneqq  \pi_W$ and $\tilde \pi\coloneqq  \pi_{W^{\perp}}$. Further observe that all $H_i$ are closed sets. 
	 Choose $j_0$ such that  $\kappa(0) \in H_{j_0}$. 
	Set $u_0\coloneqq 0$ and $u_1\coloneqq  \max\{u\in [0,1]:\ \kappa(u) \in H_{j_0}\}$. Then clearly 
	either $u_1 = 1$ or
	$\kappa(u_1) \in H_{j_1}$
      for some $j_1 \neq j_0$. Then we define $u_2\coloneqq  \max\{u\in [0,1]:\ \kappa(u) \in H_{j_1}\}$, and so on.
       By this procedure
we obtain numbers  $0=u_0 <u_1  \dots < u_q = 1$ with $1 \leq q \leq p$ and pairwise different
 indexes $j_0,j_1,\dots,j_{q-1}$ such that $\{\kappa(u_k), \kappa(u_{k+1})\} \subset H_{j_k}$
 for each $0\leq k \leq q-1$. Since clearly  $\sum_{k=0}^{q-1} |\tilde{\pi}(\kappa(u_{k+1}))-
 \tilde{\pi}(\kappa(u_{k}))| \geq |z_1- z_0|$, we can choose  $0\leq k \leq q-1$ for which
$|\tilde{\pi}(\kappa(u_{k+1}))-
 \tilde{\pi}(\kappa(u_{k}))| \geq (1/p) |z_1- z_0|$. Since
$$ \kappa(u_{k}) = \pi(\kappa(u_{k})) + \psi_{j_k}(\pi(\kappa(u_{k})))\ \ \text{and}\ \ 
 \kappa(u_{k+1}) = \pi(\kappa(u_{k+1})) + \psi_{j_k}(\pi(\kappa(u_{k+1}))),$$
  we have
$$ |\tilde{\pi}(\kappa(u_{k+1}))-
 \tilde{\pi}(\kappa(u_{k}))|= |\psi_{j_k}(\pi(\kappa(u_{k+1}))) - \psi_{j_k}(\pi(\kappa(u_{k})))|
 \leq \omega |\pi(\kappa(u_{k+1})) - \pi(\kappa(u_{k}))|.$$
Consequently we obtain \eqref{odhdiam}, since
$$ |\kappa(u_{k+1}) - \kappa(u_{k})| \geq |\pi(\kappa(u_{k+1})) - \pi(\kappa(u_{k}))| 
\geq \frac{1}{p \omega} |z_1- z_0|.$$
	\end{proof}
	
\begin{proposition}\label{hlavni}
Let $M \subset \R^d$ be a compact WDC set and $A$ a  nonempty relatively open subset of $M$. Let $A \subset P_1\cup\dots \cup P_s\eqqcolon P$, where all $P_i$
 are DC surfaces of dimension $0<k<d$. Then $A \setminus A^{[k]}$ can be covered by finitely many
 DC surfaces of dimension $k-1$.
\end{proposition}
\begin{proof}
Let $f=f_M$ and $0<\ep=\ep_M<1$ be as in Notation \ref{fixaury}.
 Choose a set $E$ by Lemma \ref{transv} corresponding to $\omega\coloneqq  \frac{\ep}{8s}$. 
 Further choose a set $T$ by Lemma \ref{Z} corresponding to $c\coloneqq  \ep/5$. Since the 
 set $E \cup T$ is a finite union of DC surfaces of dimension $k-1$, it is sufficient
 to prove  $A \setminus (E \cup T) \subset A^{[k]}$.

To this end, choose an arbitrary $x \in A \setminus (E \cup T)$. By the choice of $E$, 
 there exist $\eta(x)>0$ and  $W \in G(d,k)$  such that,  denoting $U\coloneqq  B(x,\eta(x))$,  conditions
 (i) and (ii) of Lemma \ref{transv} hold. So, if $x \in P_i$, we can choose corresponding  $D_i$ and $\psi_i$. We can and will suppose that, for some $1\leq p\leq s$, $\{i: x \in P_i\}
 = \{1,\dots,p\}$. Recall that (see Lemma \ref{transv}~(ii))
\begin{equation}\label{lippsi}
\psi_i\ \ \text{is}\  \omega-\text{Lipschitz\ on}\ D_i,\ \ \ i=1,\dots,p.
\end{equation}

   Denote $\pi\coloneqq  \pi_W$, $\tilde \pi\coloneqq  \pi_{W^{\perp}}$ and $t_0\coloneqq  \pi(x)$.
	Since $T$ is closed, we can choose $\rho \in (0,\ep)$ 
	 so small that
	\begin{equation}\label{mimoet}
\overline{B}(x, \rho) \subset  U \setminus  T, 
\end{equation}
\begin{equation}\label{ska}
	\text{$\overline{B}(t_0, \rho) \cap W \subset D_i$ for $1 \leq i \leq p$ and}
	\end{equation}
	\begin{equation}\label{marho}
	  M \cap \overline{B}(x, \rho) = A \cap  \overline{B}(x, \rho).
	\end{equation}
Note that
\begin{equation}\label{mhvez}
   M \cap \overline{B}(x, \rho) = A \cap  \overline{B}(x, \rho)\subset A \cap U \subset
 P_1\cup\dots\cup P_p. 
\end{equation}
Choose $\delta \in (0,\rho)$ such that
	\begin{equation}\label{voldel}
	\delta \left(2 + \frac{6}{\ep}\right) < \rho.
	\end{equation}
	We will now prove that 
   \begin{equation}\label{jedn}
\text{  for each  $t \in \overline{B}(t_0,\delta)\cap W$  there is at most one $x \in A \cap U$ with $\pi(x)=t$.}
   \end{equation}
   So suppose to the contrary that there exist $t \in \overline{B}(t_0,\delta)\cap W$  and $x_1, x_2 \in A \cap U$ such that
   $x_1 \neq x_2$ and $\pi(x_1)= \pi(x_2)= t$. Using \eqref{mhvez}, we can suppose that $\overline{x_1,x_2} \cap A = \{x_1, x_2\}$.
    Choose indices $i_1 \neq i_2$ such that $x_1 = t + \psi_{i_1}(t)$ and  $x_2 = t + \psi_{i_2}(t)$.  By \eqref{lippsi} we obtain
    $$ |\psi_{i_k}(t) - \psi_{i_k}(t_0)| \leq (1/8)|t-t_0| \leq (1/8)\delta ,\ \ \ k=1,2. $$
    Thus  $|x_1-x_2| =  |\psi_{i_1}(t) - \psi_{i_2}(t)|\leq \delta < \ep$.
		Now set $\vf(u)\coloneqq  x_1 + u (x_2-x_1),\ u \in [0,1]$. 
		 By Lemma \ref{L:connecting} 
		there exists a continuous $\kappa: [0,1] \to \partial M$ for which
		\begin{equation}\label{malydia}
\kappa(0)=x_1,\ \kappa(1)=x_2,\ \ \text{and}\ \ \diam(\kappa([0,1])) < \frac{6 |x_1-x_2|}{\ep}.
\end{equation}
Using also \eqref{lippsi}  and \eqref{voldel}, for each $0\leq u \leq 1$ we obtain
\begin{equation}\label{karho}
 |\kappa(u) - x| \leq |x_1 -x| + \diam(\kappa([0,1])) < \delta + (1/8) \delta + \frac{6\delta}{\ep} <
 \rho,
\end{equation}
 and so $\pi(\kappa(u)) \in  B(t_0,\rho)\cap W$ and $\kappa(u) \in P_1\cup\dots\cup P_p$
 by \eqref{mhvez}. Therefore,
denoting  $D\coloneqq  \overline{B}(t_0,\rho)\cap W$ and  $H_i\coloneqq  \{(t,\psi_i(t)):\ t \in D\}$, $i=1,\dots,p$,
 we have   $\kappa([0,1]) \subset H_1\cup\dots\cup H_p$ by \eqref{ska}.  
 Thus we can use Lemma \ref{71h} and obtain  $\diam(\kappa([0,1]))  \geq  |x_1-x_2|/(p \omega)  \geq  8 |x_1-x_2|/\ep$, which contradics \eqref{malydia}.

Now, setting $M^*\coloneqq  M \cap  \overline{B}(x,\delta) = A \cap  \overline{B}(x,\delta)\subset A \cap U$, we will show that
\begin{equation}\label{surj}
B(t_0,\delta/8) \cap W \subset \pi(M^*).
\end{equation}
 So suppose, to the contrary, that
 there exists $t \in (B(t_0, \delta/8)\cap W) \setminus \pi(M^*)$. Since $\pi(M^*)$ is compact and
 $t_0 \in \pi(M^*)$, we can choose $z \in \pi(M^*) \cap B(t_0, \delta/4)$ 
 such that $|z-t|= \dist(t, \pi(M^*))$. By \eqref{mhvez} we can find
	 $1 \leq j \leq p$ such that $y\coloneqq  z + \psi_j(z) \in M^*$. Setting  $z_n\coloneqq  z + (t-z)/n$, we have clearly
  $|z_n-z|= \dist(z_n, \pi(M^*))$.
	By \eqref{ska} $z_n \in B(t_0,\delta/4)\cap W
	 \subset D_j$
 and consequently  we can define  $y_n: = z_n + \psi_j(z_n) \in P_j$.
 Observe that
$$ |y-x| \leq  |z-t_0| + |\psi_j(z)-\psi_j(t_0)| < \delta/4 + (1/8) \delta/4 < \delta/2.$$
 It easily implies that, for all sufficiently large $n$,
 we have  
	$$\dist(y_n,M) = \dist(y_n, M^*) \geq \dist(z_n, \pi(M^*)) = |z_n-z|,$$
	 and so  $f(y_n) \geq  (\ep/2) |z_n-z|$  by Lemma \ref{DL}~(v). Further
	$$ |y_n-y| \leq |z_n-z| + |\psi_j(z_n)- \psi_j(z)| \leq 2 |z_n-z|.$$
	
	Consequently  $\limsup_{n \to \infty} f(y_n)/|y_n-y| \geq \ep/4$, which contradicts 
	 the choice of $T$, since $y_n \in P$, $y_n \to y$, $y \in A\setminus T$ by \eqref{mimoet} and $M^* \subset A$. Thus, \eqref{surj} is proved.
	
	Using \eqref{jedn} and \eqref{surj}, we obtain that there exists
 a uniquely determined function $\vf: B(t_0, \delta/8) \cap W \to W^{\perp}$ such that
 $u+ \vf(u) \in  A \cap U$ for each $u \in B(t_0, \delta/8) \cap W$. 
  Since $u+ \vf(u) \in  M^*$ for each $u \in B(t_0, \delta/8)\cap W$ and $M^*$ is compact, we  obtain
	 that $\vf$ is continuous. (Indeed, $g\coloneqq \pi|_{M^*}$ is continuous and injective by \eqref{jedn}, \eqref{mhvez}, and so
	 $g^{-1}$ is continuous. Now use that $u+ \vf(u) = g^{-1}(u),\ u \in B(t_0, \delta/8)\cap W$.)
	
	Since $\vf(u) \in \{\psi_1(u),\dots, \psi_p(u)\},\ u \in B(t_0, \delta/8)\cap W$, Lemma \ref{vldc}~(ix)
	 gives that $\vf$ is a DC mapping.
	
	Since
	$$ A \cap U \cap \pi^{-1}( B(t_0, \delta/8)\cap W) = \{u+ \vf(u):\ u \in B(t_0, \delta/8)\cap W\},$$
	 we easily obtain  $x \in A^{[k]}$.
\end{proof}

Now we can prove our  main
 results on the structure of WDC sets.

Recall that, by definition, DC surfaces of dimension $0$ in $\R^d$ are points and the only
 DC surface of dimension $d$ in $\R^d$  is $\R^d$.

\begin{theorem}\label{struc}
Let $M \subset \R^d$ be a compact WDC set and $A$ a   nonempty relatively open subset of $M$ of topological dimension $k$.
If $k=0$, then $A$ is a finite set. If $k>0$, then

\begin{enumerate}
\item[(i)] $A^{[k]} \neq \varnothing$ (and so $A^{[k]}$ is a DC manifold of dimension $k$).
\item[(ii)]  $A \setminus A^{[k]}$ can be covered by finitely many
 DC surfaces of dimension $k-1$.

\item[(iii)]  $A$ can be covered by finitely many
 DC surfaces of dimension $k$.
\end{enumerate} 
\end{theorem}
\begin{proof}
Set   $$k^*\coloneqq   \min\{0\leq s \leq d:\ A\ \ \ \text{can be covered by finitely many
 DC surfaces of dimension}\ \ \ s\}.$$
 If $k^* = 0$, then $A$  is finite and $k=k^*$.
 If $k^*>0$, then Proposition \ref{hlavni} (for $k^*<d$) and Proposition~\ref{wdcpokr} (for $k^*=d$) imply $A^{[k^*]} \neq \varnothing$.
 Consequently $k=k^*$ and so (i) and (iii) are obvious. The property (ii) follows
 from Proposition \ref{hlavni} and Proposition~\ref{wdcpokr}.
\end{proof}

\begin{remark}\label{ZFV}
\begin{enumerate}
\item[(a)]
Obviously, (i) and (ii) imply that $k= \dim_{top} A = \dim_H A$.
\item[(b)]
In the special case when $M$ has positive reach, Federer's results which are stated in \cite{Fe59}
 without a proof and are proved in \cite{RZ}
give that (if $0<k<d$),
 $A^{[k]}$ is a $C^{1,1}$ manifold. In this special case,  properties (ii) and (iii) are contained
 in \cite{RZ}; in \cite{Fe59}  it is  proved that $A \setminus A^{[k]}$ is
 countably $(k-1)$-rectifiable (which implies that $A$ is countably $k$-rectifiable).
\end{enumerate}
\end{remark}
Using Definition \ref{lwdc} (of a locally WDC set) and Proposition \ref{P_local}, it is easy to deduce
 from Theorem \ref{struc} its following ``local version''. 
Namely, properties (ii) and (iii) follow immediately and (i) is a consequence of (ii) (which implies that $\dim_{top}(A\setminus A^{[k]})\leq k-1$).

\begin{corollary}\label{struclok}
Let $M \subset \R^d$ be a closed locally WDC set and $A$ a   nonempty relatively open subset of $M$ of topological dimension $k$.
If $k=0$, then $A$ is an isolated set. If $k>0$, then
\begin{enumerate}
\item[(i)] $A^{[k]} \neq \varnothing$ (and so $A^{[k]}$ is a DC manifold of dimension $k$).
\item[(ii)]  $A \setminus A^{[k]}$ can be locally covered by finitely many
 DC surfaces of dimension $k-1$.
\item[(iii)]  $A$ can be locally covered by finitely many
 DC surfaces of dimension $k$.
\end{enumerate} 
\end{corollary}

An easy consequence of Corollary \ref{struclok} is the following result.

\begin{proposition}\label{slstr1}
Let $A$ be a  nonempty relatively open subset of a closed locally WDC set  $M \subset\R^d$. Define the ``regular part'' of $A$ as
 $A_{reg}: = \bigcup_{i=0}^d A^{[i]}$. Then $A_{reg}$ is open and dense in $A$.
\end{proposition}
\begin{proof}
It is clear that $A_{reg}$ is open in $A$. We will prove that 
\begin{equation}\label{husto}
\text{$A_{reg}$ is dense in $A$}
\end{equation}
 by
 induction on $k\coloneqq  \dim A$. The case $k=0$ is trivial by Corollary \ref{struclok}. Now suppose that a $1\leq k \leq d$
 and \eqref{husto} holds for all $A$ with $0\leq \dim A < k$. Let now $M$ and $A$ such that
 $\dim A =k$ be given.     
Set $\tilde A\coloneqq   A \setminus \overline{ A^{[k]}}$. Then $\tilde A$ is a   relatively open subset 
 both of $A$ and of $M$.  If $\tilde A = \varnothing$, then \eqref{husto} clearly holds.
 If $\tilde A \neq \varnothing$, then $\dim \tilde A < k$ by Corollary \ref{struclok}~(i). 
  So, applying the induction assumption
 to $\tilde A$, we easily obtain \eqref{husto}.
\end{proof}

\begin{remark}\label{slstr2}
 Let $M \subset\R^d$ be  a closed locally WDC set.
By  Corollary \ref{struclok}, the set  $M \setminus M_{reg}$ (which is nowhere dense in $M$ by Proposition \ref{slstr1}) has localy finite
$(\dim M-1)$-dimensional Hausdorff measure. So, if $M$ is compact, $M \setminus M_{reg}$ has 
 finite
$(\dim M-1)$-dimensional Hausdorff measure.
\end{remark}
\begin{corollary}\label{topman}
Let $M \subset \R^d$ be a closed locally WDC set which is a $k$-dimensional topological manifold ($1\leq k <d$).
 Then there exists a closed set $N\subset M$ which is nowhere dense in $M$ such that $M \setminus N$
 is a $k$-dimensional DC manifold.
\end{corollary}
\begin{proof}
Since $A^{[l]}= \varnothing$ for each $l \neq k$, the assertion follows from
 Proposition \ref{slstr1}. 
\end{proof}

\section{WDC sets in plane}  \label{plane}
In this section we will call a $1$-dimensional DC surface in $\er^2$ a {\it DC graph}.
Under a rotation (in $\er^2$) we always understand a rotation around the origin.

\begin{remark}\label{dcgr}
	Let $P\subset \R^2$ be a DC graph of the form $P = \{w+ \vf(w):\ w \in W\}$, where $W \in G(2,1)$
	and $\vf: W \to W^{\perp}$ is a Lipschitz DC mapping. Let $a=c + \vf(c) \in P$.
	Then
	\begin{enumerate}
		\item[(i)]
		$\Tan(P,a) \cap S^1  $ is a two point set,
		\item[(ii)]
		$\Tan(P,a)$ is a $1$-dimensional space iff $\vf'(c)$ exists, and
		\item[(iii)]
		there exist DC graphs $P_1, P_2 \subset \R^2$ such that $P \subset P_1 \cup P_2$, $a \in P_1\cap P_2 $
		and $\Tan(P_i,a)$ is a $1$-dimensional space, $i=1,2$.
	\end{enumerate}
	
	Indeed, without any loss of generality we can suppose that $W$ is the $x$-axis and $a=0$. Then
	$P$ is the graph of the function $\psi: \R \to \R$, $\psi(t)\coloneqq  \vf((t,0)) \cdot (0,1)$.
	Since $\psi'_{\pm}(0)$ exist by  Lemma \ref{strict}   and $\vf'(0)$ exists iff $\psi'(0)$ exists, (i) and (ii) follow 
	from well-known elementary facts.       
	
	To prove (iii), it is sufficient to define $\psi_1$ and $\psi_2$ as the odd functions, which
	coincide with $\psi$ on $[0,\infty)$ and $(-\infty,0]$, respectively, set 
	$P_1\coloneqq  \graph \psi_1$, $P_2\coloneqq  \graph \psi_2$, and use (ii) and Lemma \ref{vldc}~(iv), (vi), (ix).
\end{remark}

Further, in this section, we will say that a function $f$ defined on a set $D \subset \er^d$ is a {\it DCR function} if it is a restriction
of a DC function  defined on $\R^d$. 

Lemma \ref{vldc}~(viii) implies that
\begin{equation}\label{DCRint}
\text{if $g$ is DC on $(a,b)$ and $[c,d] \subset (a,b)$, then $g|_{[c,d]}$ is DCR.}
\end{equation}

If $z\in \R^2$ and  $v\in S^1$ we denote by   $\gamma_{z,v}$ the unique orientation preserving isometry on $\er^2$ that maps $0$ to $z$ and $(1,0)$ to $z+v$.

Further, for $u>0$, $s\in(0,\infty]$, $z\in\er^2$ and $v\in S^1$, we define
\begin{equation*}
A_s^u\coloneqq  \left\{(x,y): 0\leq x< s, -xu\leq y\leq xu \right\}\ \ \text{and}\ \ A_s^u(z,v)\coloneqq \gamma_{z,v}(A_s^u).
\end{equation*}

For $v,w \in S^1$ (possibly $v=w$), we denote by $\arc(v,w)$ the open arc on $S^1$ whose closure is the image of a simple (or closed simple) counter-clockwise orientated curve starting at $v$ and terminating at $w$. More generally, if $r>0$ and $p,q \in \partial B(0,r)$, we denote   the corresponding  arc on   $ \partial B(0,r)$
by  $\arc^r(p,q)\coloneqq  r\cdot\arc(\frac pr, \frac qr)  $.

For $B\subset \er^2$ and $t \in \R$, we set  $B_{t}\coloneqq  \{y\in \R:\ (t,y) \in B\}$.    
We also define $\Pi_1:\er^2\to\er$ by $\Pi_1(x,y)=x$.

If $K\subset \er$ and $f:K\to\er$ is a function, then $\hyp f$ and $\epi f$ will be used for hypograph and epigraph of $f$, respectively;
\begin{equation*}
\hyp f\coloneqq \{(x,y)\in\er^2: x\in K,\ y\leq f(x)\},\ \ \ \epi f\coloneqq \{(x,y)\in\er^2: x\in K, \ y\geq f(x)\}.
\end{equation*}
Similarly we define strict hypograph and strict epigraph of $f$ by
\begin{equation*}
\shyp f\coloneqq \{(x,y)\in\er^2: x\in K, \ y< f(x)\},\ \ \ \sepi f\coloneqq \{(x,y)\in\er^2: x\in K,\;y> f(x)\}.
\end{equation*}

We will need the following lemmas.

\begin{lemma}\label{mondist}
	Let $\delta>0$ and let  $f$ be a DC function on $(- \delta, \delta)$ with $f(0)=0$. Then there exists $\omega \in (0,\delta)$
	such that the function $R(x)\coloneqq  \sqrt {x^2 + f^2(x)}$ is strictly increasing on $[0,\omega)$ and 
	strictly decreasing on $(-\omega,0]$.
\end{lemma}
\begin{proof}
	By Lemma \ref{strict}, for each $0<x<\delta$, there exists
	$ R'_+(x) =  \frac{x}{R(x)} (1 + \frac{f(x)}{x} f'_+(x))$ and $\lim_{x\to 0+} (1 + \frac{f(x)}{x} f'_+(x)) =
	1 +   (f'_+(0))^2 >0$.  Consequently the (continuous) function $R$ is strictly increasing on some $[0,\omega)$. Considering the function $g(x)\coloneqq  f(-x)$, we easily obtain the rest of the lemma.
\end{proof}

\begin{lemma}\label{zestr}
	Let $P$ be a DC graph in $\R^2$ and $0 \in P$.
	Suppose that $\Tan(P,0)$ is a $1$-dimensional space and $(0,1) \notin \Tan(P,0)$. Then there exist
	$\rho^*>0$ such that, for each $0< \rho< \rho^*$, there exist $\alpha<0 < \beta$ and a DCR function $f$ on $(\alpha, \beta)$ such that
	$P \cap B(0, \rho) = \graph  f|_{(\alpha, \beta)}$.
\end{lemma}
\begin{proof}
	Let $P= \{z+ \vf(z):\ z \in W\}$, where  $W \in G(2,1)$ and $\vf: W \to  W^{\perp}$ is a Lipschitz DC mapping.
	Choose $w \in W\cap S^1$ and set $\omega(t)= (\omega_1(t), \omega_2(t)): = t w + \vf (t w),\ t \in \R$.
	Then $\omega$ is a Lipschitz DC mapping, $P= \omega(\R)$ and $\omega_1= \Pi_1 \circ \omega$. By the assumptions and Remark \ref{dcgr}~(ii), 
	$\vf'(0)$ and consequenly also $\omega'(0)$ exist and $\omega_1'(0) \neq 0$.  We can (and do)
	suppose that $\omega_1'(0) > 0$ (since otherwise we can consider $\tilde w= -w$  instead of $w$). 
	Lemma \ref{strict} implies that there exists $c> 0$ such that 
	$(\omega_1)'_+(t) > c$ for each $t \in (-c,c)$. Consequently 
	$\omega_1(v)- \omega_1(u)= \int_u^v \omega_1'(t) dt \geq c (v-u)$ whenever $-c <u <v <c$, which easily implies that $(\omega_1)^{-1}$
	exists and is $(1/c)$-Lipschitz on its domain $(\gamma, \delta)$ with $\gamma < 0 < \delta$.
	Setting $P^*\coloneqq  \omega((-c,c))$, we obtain that
	$(\Pi_1|_{P^*})^{-1} = \omega \circ   \omega_1^{-1}$ is Lipschitz. Therefore
	$P^*$ is a graph of a Lipschitz function $g: (\gamma, \delta) \to \R$, which is 
	DC by Lemma \ref{vldc}~(iv), (v), (vi), (vii).
	Now, using the obvious fact that $0 \notin  \overline{P \setminus P^*}$ and applying
		Lemma \ref{mondist} (to $g$) together with \eqref{DCRint}, it is easy to show the existence
		 of $\rho^*$ such that, for each  $0< \rho< \rho^*$, corresponding $\alpha$, $\beta$ and $f$ ($f= g|_{(\alpha,\beta)}$) exist.
\end{proof}

\begin{definition} \label{types}
	Let $M \subset \R^2$ and $r, u>0$. We say that
	\begin{enumerate}
		\item
		$M$ is a {\it $\tilde T_{r,u}^1$-set} if $M\cap A_{r}^{2u}=\{0\}$.
		\item
		$M$ is a {\it $\tilde T_{r,u}^2$-set} if   $M\supset A_{r}^{2u}$.
		\item
		$M$ is a  {\it $\tilde T_{r,u}^3$-set} if there is a DCR function $U:[0,r)\to\er$ such that $U'_{+}(0)=0$, $\graph U\subset A_{r}^{u}$ and 
		$M\cap A_{r}^{2u}=\hyp U\cap A_{r}^{2u}$.
		\item
		$M$ is a  {\it $\tilde T_{r,u}^4$-set} if  there is a DCR function $L:[0,r)\to\er$ such that $L'_{+}(0)=0$, $\graph L\subset A_{r}^{u}$ and 
		$	M\cap A_{r}^{2u}=\epi L\cap A_{r}^{2u}$.
		\item
		$M$ is a  {\it $\tilde T_{r,u}^5$-set} if there are DCR functions $U,L:[0,r)\to\er$ such that $L\leq U$ on $[0,r]$, $U'_{+}(0)=L'_{+}(0)=0$,
		$\graph U,\graph L\subset A_{r}^{u}$, and 
		$M\cap A_{r}^{2u}=\hyp U\cap\epi L$.
		\item
		$M$ is of type $T^i$ ($i=1,2,3,4,5$) at $x\in M$ in direction $v \in S^1$ if the preimage $(\gamma_{x,v})^{-1}(M)$ is a
		$\tilde T_{r,u}^1$-set for some $r,u>0$.
	\end{enumerate}
\end{definition}

\begin{remark}\label{otypech}
	\begin{enumerate}
		\item
		Clearly, if $M$ is a $\tilde T_{r,u}^i$-set (resp. of  type $T^i$  at $x$ in direction $v$), then $i$ is uniquely determined.
		\item
		If $M$ is a $\tilde T_{r,u}^i$-set and $\delta>0$, then clearly  $M$ is a $\tilde T_{r^*,u^*}^i$-set for some $r^*<\delta$, $u^* < \delta$ (note that $U'_{+}(0)=L'_{+}(0)=0$ in Definition \ref{types}). 
	\end{enumerate}
\end{remark}
\begin{remark}\label{otypech2}
	For $M \subset \R^2$ with $0 \in M$ and $1\leq i \leq 5$, set
	\begin{equation}\label{vei}
	\text{$V_i\coloneqq  \{v \in S^1:\ 	M$\  is of type\  $T^i$ \  at\ $0$\ in direction\ $v \in S^1\}$.}
	\end{equation}
	Then clearly
	\begin{enumerate}
		\item 
		$V_1$ and $V_2$ are open in $S^1$.
		\item
		If $K \subset V_1$ (resp.  $K \subset V_2$) is compact, then the ``covering definition of compactness'' easily gives that there is $\rho>0$ such that
		$K_{\rho}\coloneqq  \{tv:\ v \in K,\ 0<t<\rho\} \cap M = \varnothing$ (resp.  $K_{\rho} \subset M$). 
	\end{enumerate}
\end{remark}

\begin{lemma}\label{L:connecting2}
	Let $M\neq\varnothing$ be a compact WDC set in $\er^2$, $u,s,t>0$, $N\in\en$. Let $\eps=\eps_M$ be as in Notation~\ref{fixaury}.
	Suppose that $\partial M\cap A_s^u$ is covered by the graphs of $\frac{\ep}{12\, N}$-Lipschitz functions $f_1,\dots,f_N:\er\to\er$ and that $\partial M\cap A_s^{2u}\subset A_s^u$.
	Suppose that
	\begin{equation}\label{fourineq}
	\quad 2su<\eps,\quad \frac{12u}{\eps}<1\quad\text{and}\quad t\left(1+\frac{12u}{\eps}\right)<s.
	\end{equation}
	Then $(M\cap A_t^{2u})_z$ is connected for every $0\leq z< t$.
\end{lemma}
\begin{proof}
	Suppose for contradiction that $(M\cap A_t^{2u})_z$ is disconnected for some $0\leq z< t$.
	This implies that there are 
	\begin{equation*}
	p,q\in\partial M\cap \left(\{z\}\times (A_t^{2u})_z\right),
	\end{equation*} $p\not=q$, such that $\overline{pq}\cap M=\{p,q\}$.
	Consider $\phi:[0,1]\to \overline{pq}$ defined as $\phi(\alpha)=(1-\alpha)q+\alpha p$.
	By Lemma~\ref{L:connecting} with $\delta=|p-q|=\diam(\phi([0,1]))$ (note that $|p-q|\leq 2zu< 2tu< 2su<\eps$ by the first and third inequality in \eqref{fourineq} together with the inclusion $\partial M\cap A_s^{2u}\subset A_s^u$) there is a mapping 
	$\kappa:[0,1]\to\partial M$ such that $\kappa(0)=q$, $\kappa(1)=p$ and
	\begin{equation}\label{eq:diamUpperBound}
	\diam \kappa([0,1])\leq \frac{6\delta}{\eps}=\frac{6|p-q|}{\eps}\leq \frac{12uz}{\eps}.
	\end{equation}
	Since $\frac{12uz}{\eps}<z$ by the second inequality in \eqref{fourineq} and $z+z\frac{12u}{\eps}< t(1+\frac{12u}{\eps})<s$ by the third inequality in \eqref{fourineq},
	we infer from \eqref{eq:diamUpperBound} that
	\begin{equation*}
	\Pi_1(\kappa([0,1]))\subset \left[z-\frac{12uz}{\eps},z+\frac{12uz}{\eps}\right]\subset\left(0,s\right).
	\end{equation*}
	This easily implies (using $\partial M\cap A_s^{2u}\subset A_s^u$) that $\kappa([0,1])\subset A_s^u$.
	However, by the assumptions of the lemma together with Lemma~\ref{71h} (used with $|z_1-z_0|= \delta$ and $\omega: = \frac{\ep}{12\, N}$), we obtain $\diam(\kappa([0,1])) \geq \frac{\delta}{N \omega} = \frac{12\delta}{\eps}$, which contradicts  \eqref{eq:diamUpperBound}.
\end{proof}

\begin{lemma}\label{L:locpie}
	Let $M$ be a closed locally WDC set in $\er^2$, $x\in\partial M$ and $v\in S^1$.
	Then there exists $1\leq i \leq 5$ such that $M$ is of type $T^i$ at $x$ in direction $v$.
\end{lemma}

\begin{proof}
	We may assume (by Proposition~\ref{P_local} and Corollary~\ref{Cor:propertiesOfRetracts}) that $M$ is compact and connected.
	By Definition \ref{types} (and the fact that isometric images of locally WDC sets are locally WDC, cf.\ \cite[Theorem~7.5]{PR13}), it is sufficient to prove that if $0 \in \partial M$, then there exist $r,u>0$ and
	$1\leq i \leq 5$ such that $M$ is a $\tilde{T}_{r,u}^i$-set.
	Let $\eps_M$ be as in Notation~\ref{fixaury}.
	
	By Lemma~\ref{wdcpokr} $\partial M$ can be covered by finitely many DC graphs $P_1,\dots,P_n$.
	Put $I=\{i: 0 \in P_i  \}$.
	If $I = \varnothing$, then there clearly exist $r, u >0$ such that $\partial M\cap A_{r}^{2u}=\{0\}$
	and so $M$ is either a $\tilde{T}_{r,u}^1$-set or a a $\tilde{T}_{r,u}^2$-set.
	
	If $I \neq \varnothing$, then we can suppose that $I= \{1,\dots, N\}$. 
	Due to Remark~\ref{dcgr}~(iii) we can and will also suppose that $\Tan(P_i,0)$ is a $1$-dimensional linear subspace.
	
	Put $\tilde I=\{i: \Tan(P_i,0)=\spa \{(1,0)\} \}$.
	Again, if $\tilde I = \varnothing$, then there exist $r, u >0$ such that  $\partial M\cap A_{r}^{2u}=\{0\}$
	and so $M$ is either a $\tilde{T}_{r,u}^1$-set or a a $\tilde{T}_{r,u}^2$-set.
	If $\tilde I \neq \varnothing$, then we can suppose that $\tilde I= \{1,\dots, \tilde N\}$. 
	
	Using Lemma \ref{zestr} and Lemma \ref{mondist} we obtain that for each $1\leq i \leq \tilde N$ there exist 
	$u_i, s_i\in (0,\infty)$ and a DCR function $\vf_i$ on $[0, s_i)$ such that
	$P_i \cap A_{s_i}^{2u_i} = \graph \vf_i$. 
	Note that $(\vf_i)'_+(0) = 0$ and so, using Lemma \ref{strict}
	we  obtain 
	$s,u >0$  such that, denoting  $f_i \coloneqq  \vf_i |_{[0,s)}$, we have
	\begin{enumerate}[label={\rm (f\arabic*)}]
		\item\label{5} $(f_i)'_+(0) = 0$, $f_i(0)=0$ and $f_i$ is $\frac{\eps_M}{12\tilde N}$-Lipschitz and DCR for every $i$,
		\item\label{2} $\partial M\cap A_{s}^{u}$ is covered by  the graphs of $f_i$,
		\item\label{4} $\partial M\cap A_{s}^{2u}\subset A_{s}^{u}$,
		\item\label{1} $2su<\eps_M$ and $\frac{12u}{\eps_M}<1$.
	\end{enumerate}
	
	Using Lemma~\ref{L:connecting2} we obtain that
	\begin{equation}\label{sections}
	(M\cap A_t^{2u})_x \quad\text{is connected for every}\quad 0\leq x< t
	\end{equation} 
	provided $t$ satisfies $t(1+\frac{12u}{\eps})<s$.
	Pick any such $t$.
	
	Define functions $\tilde U,\tilde L:\Pi_1(M\cap A_t^{2u})\to\er$ by 
	$\tilde U(x)\coloneqq \max(M\cap A_t^{2u})_x$ and $\tilde L(x)\coloneqq \min(M\cap A_t^{2u})_x$, $x\in \Pi_1(M\cap A_t^{2u})$ (note that $(M\cap A_t^{2u})_x$ is always compact).
	Clearly $\tilde U\geq \tilde L$ on $\Pi_1(M\cap A_t^{2u})$.

	Put
	\begin{equation*}
	A^\pm\coloneqq  (A_t^{2u}\setminus A_{t}^{u})\cap \{(x,y):\pm y>0\}.
	\end{equation*}
	There are four possible cases
	\begin{enumerate}[label={ (\alph*)}]
		\item\label{a} $A^+\cap M=A^-\cap M=\varnothing$,
		\item\label{b} $A^+\cap M=\varnothing$ and $A^-\cap M\not=\varnothing$,
		\item\label{c} $A^+\cap M\not=\varnothing$ and $A^-\cap M=\varnothing$,
		\item\label{d} $A^+\cap M\not=\varnothing$ and $A^-\cap M\not=\varnothing$.
	\end{enumerate}
	
	Observe that by \ref{4} we have
	\begin{equation}\label{eq:bAndd}
	\tilde L(z)=-2uz,\; z\in[0,t),\quad\text{in cases \ref{b} and \ref{d}}
	\end{equation}
	and
	\begin{equation}\label{eq:cAndd}
	\tilde U(z)=2uz,\; z\in[0,t),\quad\text{in cases \ref{c} and \ref{d}}.
	\end{equation}
	Suppose that $M$ is not a $\tilde{T}_{r,u}^1$-set for any $r>0$.
	We then claim that
\begin{equation} \label{claim}	
\tilde U\text{ and }\tilde L\text{ are DCR on }[0,r)\text{ for some }r>0.
\end{equation}
	By \eqref{eq:bAndd} and \eqref{eq:cAndd} it is enough to prove \eqref{claim} for $\tilde U$ in cases \ref{a} and \ref{b} and for $\tilde L$ in cases \ref{a} and \ref{c}.
	
	Clearly
	\begin{equation}\label{eq:aAndb}
	\graph\tilde U\subset A_{t}^{u}\quad\text{in cases \ref{a} and \ref{b}}
	\end{equation}
	and
	\begin{equation}\label{eq:aAndc}
	\graph\tilde L\subset A_{t}^{u}\quad\text{in cases \ref{a} and \ref{c}}.
	\end{equation}
	
	We will consider only the function $\tilde U$ in cases \ref{a} and \ref{b}, 
	the case of  $\tilde L$ in cases \ref{a} and \ref{c} can be proven analogously.
	
	So suppose that either \ref{a} or \ref{b} holds.
	Then, using the fact that $M$ is connected, we obtain that $\Pi_1(M\cap A_t^{2u})$ has at most two connected components, 
	of which one contains an interval of the form $[0,2r]$ for some $\frac{t}{2}> r> 0$ (if not then $M$ would be a $\tilde{T}_{r,u}^1$-set).
	
	We prove that $\tilde U$ is continuous on $[0,2r]$.
	First note that $\tilde U$ is continuous at $0$ by \eqref{eq:aAndb}.
	Pick $x\in(0,2r]$.
	Since $M\cap A^{2u}_t$ is relatively closed in $A^{2u}_t$, we obtain that $\tilde U$ is upper semi-continuous at $x$, that is
	\begin{equation*}
	\limsup_{y\to x,\, y\in [0,2r]}\tilde U(y)\leq \tilde U(x).
	\end{equation*}
	It remains to prove that
	\begin{equation*}
	\liminf_{y\to x,\, y\in [0,2r]}\tilde U(y)\geq \tilde U(x).
	\end{equation*}
	To do that assume for a contradiction that there is $\eps>0$ so that 
	\begin{equation*}
	\liminf_{y\to x,\, y\in [0,2r]}\tilde U(y)= \tilde U(x)-\eps.
	\end{equation*}
	This implies that there is a sequence $y_n\to x$, $y_n\in [0,2r]$, such that $\lim_{n\to \infty }\tilde U(y_n)= \tilde U(x)-\eps$. 
	Pick $0<\alpha<\eps$. 
	Then $(y_n,\tilde U(y_n)+\alpha)\in M^c$ for sufficiently big $n$ and since  $(y_n,\tilde U(y_n)+\alpha)\to (x,\tilde U(x)-\eps+\alpha)\in M$ we obtain that $(x,\tilde U(x)-\eps+\alpha)\in\partial M\cap A_{s}^{u}$.
	But that being true for all $0<\alpha<\eps$ is a contradiction with \ref{2}.
	
	Extend $\tilde U$ by setting $\tilde{U}(x):=0$, $x\in[-2r,0)$. $\tilde{U}$ is now continuous on $[-2r,2r]$ and
	by \ref{2}, $\graph \tilde U$ is covered by the graphs of $f_i$ 
	together with the graph of $f_0\equiv 0$. Hence, 
	we obtain that $\tilde U$ is DC on $(-2r,2r)$ by Lemma~\ref{vldc}~(ix) and so $\tilde U$ is DCR on $[0,r)$ by \eqref{DCRint}.
	This proves \eqref{claim}.
	Pick some $r>0$ such that $\tilde U$ and $\tilde L$ are DCR on $[0,r)$.
	It follows that both $\tilde U'_{+}(0)$ and $\tilde L'_{+}(0)$ exist and, moreover, they are both equal to $0$ by \ref{5} and \ref{2}.
	
	Put $U\coloneqq \tilde U|_{[0,r)}$ and $L\coloneqq \tilde L|_{[0,r)}$.
	Then by \eqref{eq:bAndd}, \eqref{eq:aAndb}, \eqref{eq:cAndd} and \eqref{eq:aAndc} we easily obtain that 
	\ref{a} implies that $M$ is a $\tilde{T}_{r,u}^5$ set, 
	\ref{b} implies that $M$ is a $\tilde{T}_{r,u}^3$ set,
	\ref{c} implies that $M$ is a $\tilde{T}_{r,u}^4$ set and
	\ref{d} implies that $M$ is a $\tilde{T}_{r,u}^2$ set.
	\end{proof}

Now we will define important notions of ``DC sectors''.

\begin{definition}
	\label{sektory}
	\begin{enumerate}
		\item
		A set $S \subset \R^2$ will be called a {\it basic open DC sector} (of radius $r$) if $S= B(0,r) \cap \sepi(f)$, where  $0<r<\omega$ and $f$ is a DC function on $(-\omega, \omega)$
		such that $f(0)=0$, $R(x)\coloneqq  \sqrt {x^2 + f^2(x)}$ is strictly increasing on $[0,\omega)$ and 
		strictly decreasing on $(-\omega,0]$. 
		
		By an {\it open DC sector} (of radius $r$) we mean an image $\gamma(S)$ of a  basic open DC sector $S$ (of radius $r$) under a rotation $\gamma$.
%		\item
%		By a {\it non-degenerated closed DC sector} we mean the closure of an open DC sector.
		\item
		A set of the form $\gamma(\hyp f\cap \epi g) \cap B(0,r)$, where $\gamma$ is a rotation, $0<r<\omega$ and $f,g:[0,\omega)\to\er$ are DCR functions such that $g \leq f$, $f(0)=g(0)=f'_+(0)=g'_+(0)=0$ and
		the functions $R_f(x)\coloneqq  \sqrt {x^2 + f^2(x)}$, $R_g(x)\coloneqq  \sqrt {x^2 + g^2(x)}$ are strictly increasing on $[0,\omega)$,  will be called a
		{\it
			degenerated closed DC sector} (of radius $r$).
	\end{enumerate}
\end{definition}

\begin{remark}\label{otsec}
	\begin{enumerate}
		\item[(a)]
		If   $f$ is a DC function on $(-\omega, \omega)$ with  $f(0)=0$, then Lemma \ref{mondist} implies that there  exists $r>0$ such that $S= B(0,r) \cap \sepi(f)$ is a basic open DC sector.     
		\item [(b)] 
		Let $S$ be a basic open DC sector and $0<r < \omega$ and $f$ be as in Definition \ref{sektory}.
		Then
		\begin{enumerate}
			\item [(b1)] 
			For each $0<r^*<r$, $S \cap B(0, r^*)$ is a basic open DC sector.
			\item[(b2)] 
			There clearly exist $- r\leq a < 0 < b \leq r$ such that $(t, f(t)) \in \ B(0, r)$ for $t \in (a,b)$,
			$(t, f(t)) \in \ \partial B(0, r)$ for $t\in \{a,b\}$ and  
			$(t, f(t)) \notin \ \overline B(0, r)$ for $t \in (-\omega, \omega) \setminus [a,b]$. Distinguishing the cases when
			$f(a)\geq 0$ and $f(a) < 0$ ($f(b)\geq 0$ and $f(b) < 0$), we easily see  that in all possible four cases $S$ is an open connected set with $\partial S = \graph (f|_{[a,b]}) \cup
			\arc^r((b,f(b)), (a, f(a)))$.
		\end{enumerate}
	\end{enumerate}
\end{remark}
We will need the following lemma, whose elementary proof will be ommited.
\begin{lemma}\label{prist}
	Let $s,u>0$, $0<r<\omega$ and let $h$ be a DCR function on $[0,\omega]$.
		Suppose that $h(0)=0$, $h'(0)=0$, $\graph h \subset A_s^u$ and the function
	$R_h(x)\coloneqq  \sqrt {x^2 + h^2(x)}$ is strictly increasing on $[0,\omega]$. 
	Let $\arc^r(p,q)\coloneqq  A_s^{2u} \cap rS^1$.
		Then there exists  $0 < c \leq r$ such that $(t, h(t)) \in \ B(0, r)$ for $t \in [0,c)$,
	$(c, h(c)) \in \ \partial B(0, r)$ and  
	$(t, h(t)) \notin \ \overline B(0, r)$ for $t \in  (c,\omega)$. Further,
	\begin{enumerate}
		\item[{\rm (i)}]\ the set \ $P^+\coloneqq  \sepi h \cap \INt A_s^{2u} \cap B(0,r)$\ is open and connected
		and
		$$ \partial P^+ = \graph h_{[0,c]} \cup \arc^r((c,h(c)),q) \cup \overline{0q},$$
		\item[{\rm (ii)}]\ the set \ $P^-\coloneqq  \shyp h \cap \INt A_s^{2u} \cap B(0,r)$\ is open and connected
		and
		$$ \partial P^- = \graph h_{[0,c]} \cup \arc^r(p, (c,h(c))) \cup\overline{0p}.$$
	\end{enumerate}
\end{lemma}
\begin{lemma}\label{L:extension}
	Let $\omega>0$.
	Let $f$ be a DC function on $(-\omega, \omega)$ and let $g,h$ be DCR functions on $[0,\omega)$ such that $g(0)=h(0)=g'_+(0)=h'_+(0)=0$ and  $g\leq h$ on $[0,\omega)$.
	Then
	\begin{enumerate}[label={\rm (\roman*)}]
		\item\label{cond:extensionSubgraph} there is a DC aura $F$ in $(-\omega, \omega)\times\er$ for $\hyp f$,
		\item\label{cond:extensionNotSubgraph} there is a DC aura $G$ in $(-\omega, \omega)\times\er$ for $\hyp h\cap \epi g$.
	\end{enumerate}
\end{lemma}

\begin{proof}
	To prove \ref{cond:extensionSubgraph} consider the function $F:(-\omega, \omega)\times\er\to\er$ defined by
	$F(x,y)= \max(y-f(x),0)$.
	First note that $F$ is DC by Lemma~\ref{vldc}~(i).
	Moreover, if $y > f(x)$ and $x$ is a point of differentiability of $F$ then the second coordinate of $\nabla F(x,y)$ is equal to $1$.
	Hence from the definition of the Clarke subgradient (which is recalled in the preliminaries) it follows that the second coordinate of every $v\in\partial F(x,y)$ is also equal to $1$ whenever $y>f(x)$.
	This in particular implies that $0$ is a weakly regular value of $F$ 
	and so \ref{cond:extensionSubgraph} holds, since $\hyp f=F^{-1}(\{0\})$.

	To prove \ref{cond:extensionNotSubgraph} consider the function $G:D_G=(-\omega,\omega)\times\er\to\er$ defined by
	\begin{equation*}
	G(x,y)= 
	\begin{cases} 
	g(x)-y & \mbox{if }   x\geq 0 \mbox{ and }  y\leq g(x),  \\
	y-h(x) & \mbox{if } x\geq 0 \mbox{ and }  y\geq h(x),  \\
	0 & \mbox{if } x\geq 0 \mbox{ and } g(x)> y> h(x),  \\
	\sqrt{x^2+y^2} & \mbox{otherwise.}
	\end{cases} 
	\end{equation*}
	The fact that $G$ is DC follows from Lemma~\ref{vldc}~(i), (ix).
	
	To prove the weak regularity of $0$, first observe that if $(u,v) \in D_G$, $u \geq 0$, $G$  is differentiable at $(u,v)$
	 and $v> h(u)$ (resp. $v < g(u)$), then the second coordinate of $\nabla G(u,v)$ is equal to $1$ 
	(resp. to $-1$). 
	Further consider a point $w=(x,y) \in D_G \setminus (\hyp h\cap \epi g) $ and  $\nu\in\partial G(w)$. 
	If $x>0$, then the above observation implies (as in part (i)) that $|\nu| \geq 1$.
	If $x<0$ then  $\nu =\nabla G(w)=\frac{w}{|w|}$ and so $|\nu|=1$. 
	Now consider the case $x=0$, $y>0$. Note that if $u<0$ and $v>-u$ then the second coordinate of $\nabla G(u,v)$ is at least $\frac{1}{\sqrt{2}}$. Consequently we easily obtain that 
	 the second coordinate of $\nu$ is at least $\frac{1}{\sqrt{2}}$ and thus $|\nu|\geq \frac{1}{\sqrt{2}}$.
	 By a quite symmetrical way we we obtain $|\nu|\geq \frac{1}{\sqrt{2}}$ if  $x=0$, $y<0$.
	 Since $\hyp h\cap \epi g=G^{-1}(\{0\})$, we obtain that $0$ is a weakly regular value of 
	 $G$, and thus (ii) is proved.
	\end{proof}

\begin{lemma}\label{L:twoSets}
	Let $U,V\subset \er^d$ be connected open sets such that $\partial U=\partial V$ and $U\cap V\not=\varnothing$.
	Then $U=V$.
\end{lemma}
\begin{proof}
	Aiming for a contradiction suppose that $U\not= V$. 
	Then at least one of the conditions $U\setminus V\not=\varnothing$ and $V\setminus U\not=\varnothing$ holds.
	Since $\partial U=\partial V$ we have $U\setminus V=U\setminus \overline V$ and $V\setminus U=V\setminus \overline U$.
	This  implies that the connected set $U\cup V$ can be expressed as a union of three pairwise disjoint open sets $(U\setminus \overline V)\cup(V\setminus \overline U)\cup (U\cap V)$
	of which at least two are non-empty. This is a contradiction with the connectedness of $U\cup V$.
\end{proof}

\begin{theorem}\label{T:characterisationInPlaneNew}
	Let $M$ be a closed subset of $\er^2$. Then $M$ is a  locally WDC set if and only if for each $x\in\partial M$
	there is $\rho>0$ such that one of the following conditions holds:
	\begin{enumerate}[label={\rm (\roman*)}]
		\item\label{cond:singlePointNew} $M\cap B(x,\rho)=\{x\}$,
		\item\label{cond:oneConeNew} there is a degenerated closed DC sector $C$ of radius $\rho$ such that $$M\cap B(x,\rho)=x+C,$$
		\item\label{cond:manyConesNew} there are pairwise disjoint  open DC sectors $C_1,\dots,C_k$ of radius $\rho$ such that 
		\begin{equation}\label{kolac}
		M^c\cap B(x,\rho)=\bigcup_{i=1}^k \left(x+C_i\right).
		\end{equation}
	\end{enumerate}
\end{theorem}

\begin{proof}
	First suppose that $M$ is a locally WDC set and $x \in \partial M$. We can and will suppose
	$x=0$. For $1\leq i \leq 5$, define $V_i \subset S^1$ by \eqref{vei} and observe
	that $S^1 = \bigcup_{i=1}^5 V_i$ by Lemma \ref{L:locpie}. By Proposition \ref{wdcpokr} and
	Remark~\ref{dcgr}~(i), the set $T\coloneqq  S^1 \cap \Tan(\partial M,0)$ is finite.
	Since clearly $V_3 \cup V_4 \cup V_5\subset T$, we obtain that $V_3 \cup V_4 \cup V_5$
	is finite. By Remark~\ref{otypech2}~(1) $V_1$ and $V_2$ are open in $S^1$. 
	The case $V_1 = \varnothing$ is impossible, since then would be $V_3=V_4= V_5= \varnothing$,
	so $V_2= S^1$ and thus $0 \in \INt M$ by Remark~\ref{otypech2}~(2). The above observations
	easily imply that (in the space $S^1$) the open set $\varnothing \neq V_1$ has finite number of components and thus either $V_1= S^1$ or $V_1 = \bigcup_{i=1}^k \arc (v_i,w_i)$, where the arcs 
	$\arc (v_i,w_i)$
	are pairwise disjoint. If $V_1= S^1$, we obtain (i) by Remark \ref{otypech2}~(2).
	If  $V_1\neq S^1$, then either
	$\ V_1= S^1\setminus \{v\}$
	for some $v \in S^1$, or $v_i \neq w_i$,\ $i=1,\dots,k$.
	
	If  $\ V_1= S^1\setminus \{v\}$, then the condition (ii) holds. Indeed, we can and will suppose $v= (1,0)$.
	Then there exist $r,u>0$ such that $M$ is a $\tilde{T}_{r,u}^5$-set.  
	Let $U,L: [0,r) \to \R$ be the corresponding DCR functions from Definitions \ref{types}~(5).
	Let $p,q \in S^1$ be such that $A_\infty^{2u} \cap S^1 = \arc(p,q)$. Then, applying Remark \ref{otypech2}~(2) to $K= \overline{\arc(q,p)}$ and using Lemma \ref{mondist}, we easily obtain (ii).
	
	Now  suppose that  $V_1 = \bigcup_{i=1}^k \arc (v_i,w_i)$ with $v_i \neq w_i$. 
	Clearly 
	$$K\coloneqq  V_3 \cup V_4 \cup V_5 = \{v_i: 1\leq i \leq k\} \cup \{w_i: 1\leq i \leq k\}.$$
	Further 
	\begin{equation}\label{krajni}
	z\in\begin{cases}
	V_5 &\text{ iff }z=v_i = w_j \text{ for some }i\neq j,\\
	V_3 &\text{ iff } z\notin V_5 \text{ and }z=v_i \text{ for some } i,\\
	V_4 &\text{ iff } z\notin V_5 \text{ and }z=w_i \text{ for some } i.
	\end{cases}
	\end{equation}
	Let $\xi>0$. By the definition of $V_j$ and Remark \ref{otypech}~(2) we can,
	for each $3\leq j \leq 5$ and $z \in V_j$, choose $r(z), u(z) > 0$ such that  $u(z) < \xi$ and  $(\gamma_{0,z})^{-1} (M)$
	is a   $\tilde{T}_{r(z),u(z)}^j$-set, and denote
	$$  D(z)\coloneqq  A_{\infty}^{2 u(z)} (0,z)=  \gamma_{0,z} (A_{\infty}^{2 u(z)}),\ \ \ z \in K.$$
	We can (and will) fix $\xi>0$ so small, that the angles $D(z),\ z \in K$, are pairwise disjoint.
	Now define $v_i^-,v_i^+,w_i^-,w_i^+\in S^1$ so that
	$$ \arc(v_i^-,v_i^+)= \INt D(v_i) \cap S^1,\ \ \  \arc(w_i^-,w_i^+)= \INt D(w_i) \cap S^1.$$
	Since
	$$  \overline{\arc(v_i^+, w_i^-)} \subset V_1,\ 1\leq i\leq k\ \ \text{and}\ \ S^1 \setminus \bigcup_{i=1}^k \arc(v_i^-, w_i^+)
	\subset V_2,$$
	Remark \ref{otypech2}~(2) implies that there exists $\rho^1>0$ such that, for each  $0< \rho< \rho^1$,   
	$$
	E_i = E_i^{\rho} \coloneqq \{tu:\, 0<t<\rho,\, u\in\overline{\arc(v_i^+,w_i^-)}\}\subset M^c$$
	and
	$$\bigcup_i\{tu:\, 0<t<\rho,\, u\in\overline{\arc(v_i^-,w_i^+)}\}\supset M^c.$$
	Consequently, for all  $0<\rho< \rho^1$,
	\begin{equation}\label{kdedop}
	M^c \cap B(0, \rho) = \bigcup_{i=1}^k E_i \cup  \bigcup_{i=1}^k  (M^c \cap D(v_i) \cap B(0, \rho)) \cup
	\bigcup_{i=1}^k  (M^c \cap D(w_i) \cap B(0, \rho)).
	\end{equation}
	If $z \in V_3 \cup V_5$, we define the function 
	$U^z:[0,r(z)] \to \R$ as in Definition \ref{types} and also $Q^-(z)\coloneqq  A_{r(z)}^{2u(z)} \cap \sepi U^z$ and $P^-(z): =\gamma_{0,z}(Q^-(z))$; similarly if $z \in V_4 \cup V_5$, we define the function $L^z:[0,r(z)] \to \R$,
	$Q^+(z)\coloneqq  A_{r(z)}^{2u(z)} \cap \shyp L^z$   and  $P^+(z): =\gamma_{0,z}(Q^+(z))$.
	 Moreover, using Lemma \ref{mondist}, we can suppose that the functions
	\begin{equation}\label{Mor}
	x \mapsto \sqrt{x^2 + (U^z(x))^2},\ x \mapsto \sqrt{x^2 + (L^z(x))^2}\ \ \text{are strictly increasing}.
	\end{equation}
	
	We will show that \eqref{kolac} holds for all sufficiently small $\rho>0$, if we set
	$$ C_i = C_i^{\rho} \coloneqq  E_i \cup (P^-(v_i)\cap B(0,\rho)) \cup (P^+(w_i)\cap B(0,\rho)),\ \ \ 1\leq i \leq k.$$ 
	First observe, that $C_i$ are pairwise disjoint.
	Further observe that, for each $0< \rho < \rho^1$,  the inclusion
	$M^c\cap B(x,\rho)\subset \bigcup_{i=1}^k C_i$ follows
	from \eqref{krajni} and \eqref{kdedop}, and the opposite inclusion is obvious.
	Thus it is sufficient to show that, for all sufficiently small $\rho$, each $C_i$ is an open DC sector.

	To this end, fix $1 \leq i \leq k$, consider a $0 < \rho < \rho^1$, denote  $v\coloneqq  v_i$, $w\coloneqq  w_i$, $S\coloneqq C_i$ and
	define $c$ as the midpoint of the arc $\arc(v,w)$.
	We can and will suppose that $c= (0,1)$. 
	
	Since $U^v$ is a DCR function and $U^v(0)=0$, Lemma \ref{vldc}~(ix) implies that it can be extended to a DC function 
	$U^v_*: \R \to \R$ with $U^v_*(t)=0,\ t\leq 0$. Then $P\coloneqq  \gamma_{0,v} (\graph U^v_*)$ is clearly a DC graph and $\Tan(P,0)= \spa \{v\}$.  
	 Therefore (since $\Pi_1(v)\neq 0$) we can use Lemma \ref{zestr} and choose  $0< \rho^2 < \rho^1$
such that, for each $0< \rho < \rho^2$, there exist 		
	 $\beta>0$ and a DCR function $\vf$ on $[0,\beta)$ such that
	$$A_{\infty}^{2u(v)}(0,v) \cap P \cap B(0,\rho) = \graph \vf.$$
	So, by definition of $f$, there exists  $0< \rho^3 < \rho^2$
such that, for each $0< \rho < \rho^3$,
	\begin{equation}\label{grotv}
	A_{\infty}^{2u(v)}(0,v) \cap   \gamma_{0,v}(\graph U^v) \cap B(0,\rho) = \graph \vf.
	\end{equation}
	Quite similarly we can find   $0< \rho^4 < \rho^3$
such that, for each $0< \rho < \rho^4$, there exist 		
	 $\alpha < 0$ and a DCR function $\psi$ on $(\alpha, 0]$ such that
	\begin{equation}\label{grotv2}
	A_{\infty}^{2u(w)}(0,w) \cap   \gamma_{0,w}(\graph L^w) \cap B(0,\rho) = \graph \vf.
	\end{equation}  
	Choose $\omega < \min(\beta, - \alpha)$ and set $f(t)= \vf(t)$ for  $t \in [0,\omega)$ and
	$f(t)= \psi(t)$ for  $t \in (-\omega, 0]$. Then  $f$ is by Lemma \ref{vldc} (ix) 
	a DC function on $(-\omega, \omega)$
	and so, by Remark \ref{otsec}~(a), (b1), there exists   $0< \rho^5 < \rho^4$ such that, for all 
	$0< \rho < \rho^5$, $S^*\coloneqq  \sepi(f) \cap
	B(0,\rho)$ is a (basic) open DC sector. So it is sufficient to prove that   (for all 
	$0< \rho < \rho^5$) $S=S^*$.
	
	Using Lemma \ref{prist}, it is easy to prove that $S$ is open connected, $S\cap S^*\neq \varnothing$
	and, using also \eqref{grotv}, \eqref{grotv2} and Remark \ref{otsec}~(b2)),  $\partial S = \partial S^*$. Consequently $S=S^*$ by Lemma \ref{L:twoSets}.  \smallskip

	To prove the opposite implication assume that for each $x\in\partial M$
	there is $\rho>0$ such that one of conditions \ref{cond:singlePointNew}, \ref{cond:oneConeNew} or \ref{cond:manyConesNew} holds.
	We will construct for every such $x$ and $\rho$ a DC aura $F$ in $B(x,\rho)$ for $M$.
 	This is enough by Proposition~\ref{P_loc_aura}.
	Without any loss of generality we may and will assume that $x=0$.

	In the case of condition \ref{cond:singlePointNew} we can set $F(y)=|y|$, $y\in B(0,\rho)$.
	In the case of condition \ref{cond:oneConeNew}, $M\cap B(0,\rho)$ is a degenerated closed DC sector and so there are $\omega>\rho>0$, DCR functions $ g, h:[0,\omega)\to\er$ and rotation $\gamma$ such that $ g \leq  h$ and  $ g(0)= h(0)= g'_+(0)= h'_+(0)=0$ and such that 
	\begin{equation*}
	M\cap B(0,\rho)=\gamma(\hyp h\cap \epi g)\cap B(x,\rho).
	\end{equation*}
	By Lemma~\ref{L:extension}  \ref{cond:extensionNotSubgraph} there is a DC aura $\tilde F$ in $(-\omega, \omega)\times\er$ for $\hyp h\cap \epi g$.
	Put $F\coloneqq \tilde F\circ(\gamma^{-1}|_{B(0,\rho)})$.
	Then 
	\begin{equation*}
	F^{-1}(\{0\})=B(0,\rho)\cap\gamma(\hyp f\cap \epi g)=B(0,\rho)\cap M.
	\end{equation*}
	By Lemma~\ref{vldc}~(vi), (iv) $F$ is DC and therefore a DC aura in $B(0,\rho)$ for $M$ by \cite[Theorem~2.3.10]{Cl90}.
	
	In the case of condition \ref{cond:manyConesNew}, $C_1,\dots,C_k$ are open DC sectors and therefore there are $0<r<\omega$, DC functions $f_1,\dots,f_k:(-\omega, \omega)\to\er$ such that
	$f_i(0)=0$, $i=1,\dots,k$, and rotations $\gamma_1,\dots,\gamma_k$ such that $C_i= B(0,\rho) \cap \gamma_i(\sepi f_i)$.
	By Lemma~\ref{L:extension} \ref{cond:extensionSubgraph} there are DC auras $\tilde F_1,\dots,\tilde F_k$ in $(-\omega, \omega)\times\er$ for $\hyp f_1,\dots,\hyp f_k$, $i=1\dots,k$. 
	Put $F_i\coloneqq \tilde F_i\circ(\gamma_i^{-1}|_{B(0,\rho)})$, $i=1\dots,k$.
	As above, we obtain that each $F_i$ is a DC aura in $B(0,\rho)$ for $(C_i)^c\cap B(0,\rho)$.
	Put $F\coloneqq \max_{i} F_i$; $F$ is DC on $B(0,\rho)$ by Lemma~\ref{vldc}~(i). 
	Since $C_i$ are pairwise disjoint, we have $F=F_i>0$ on $C_i$ and $F^{-1}(\{0\})=M\cap B(0,\rho)$.
	So $F$ is a DC aura in $B(0,\rho)$ for $M$.
\end{proof}

\section{Open questions}

\begin{question}  \label{Q:LR}\rm
Let $f$ be as in Lemma~\ref{DL} and assume that $f$ is, moreover, DC. Does there exist $0<\ep<1$ and a {\it Lipschitz} mapping $H:\{f<\ep\}\times[0,1]\to\{f<\ep\}$ with the properties (i) - (vii)? (Cf.\ also the comment before Example~\ref{E:retractionNotLipschitz}.)
\end{question}

\begin{question}  \rm
Is the assertion of Proposition~\ref{lipman} still valid if we assume that the closed set $A\subset\R^d$ is merely a {\it topological} manifold of dimension $0<k<d$? (Cf.\ Corollary~\ref{topman}.) In $\R^2$, the answer is positive, which follows from Theorem~\ref{T:characterisationInPlaneNew}.
\end{question}

\begin{question} \rm  \label{Q:conn}
Let $M\subset\R^d$ be a compact, connected locally WDC set. Can any two points of $x,y\in M$ be connected by a (i) rectifiable curve, or even (ii) curve with finite turn, lying in $M$? (Note that a rectifiable curve has finite turn if and only if its arc-length parametrization has DCR components, see \cite[Remark~1.1, Lemma~5.5 and Corollary~5.8]{Duda}.)
Theorem~\ref{T:characterisationInPlaneNew} implies that the answer even for (ii) is positive in $\R^2$. Note also that a positive answer to Question~\ref{Q:LR} would imply a positive answer to (i) here. In the special case $\reach M>0$, Lytchak \cite[Theorems~1.2,1.3]{Ly05} showed that the curve $\varphi$ can be found even $C^{1,1}$.
\end{question}

\begin{question} \rm
Let $M\subset\R^d$ be a compact, connected, locally WDC set, and let $x,y\in\partial M$ be two points lying in the same component of $\partial M$. Can $x,y$ be connected by a (i) rectifiable turn, (ii) curve with finite turn, lying in $\partial M$? Again, Theorem~\ref{T:characterisationInPlaneNew} implies that the answer to (ii) is positive in the planar case.
\end{question}

\begin{question}  \label{Q:decomposition}\rm
Let $M$ be a compact WDC set in $\R^d$. Does there exists a decomposition
\begin{equation}\label{decomp} M = T_1 \cup \dots \cup T_m,
\end{equation}
where $T_i$, $i=1,\dots,m$, are pairwise disjoint and each $T_i$ is a   DC manifold
of dimension $0 \leq k_i \leq d$?
\end{question}
\begin{remark}\label{R:remarkToDecomposition}
	\begin{enumerate}
		\item[(i)]
		As already mentioned in the introduction, we conjecture that the answer is positive. 
		\item[(ii)]
		The answer is positive for  $d=2$ (see  below), but the question is open for $d\geq 3$, even in the case when $M$ is a set of positive reach.
		To prove the conjecture in dimension $2$ (following the strategy from \cite[Remark~7.11~(i)]{RZ}) first observe that if $M$ has such decomposition and $K$ is closed then $M\setminus K$ has such decomposition as well. Moreover, if $M_1,\dots, M_N$ are closed and all have such decomposition then also their union $M_1\cup\dots\cup M_N$ has the decomposition since it can be expressed as a disjoint union $M_1\cup (M_2\setminus M_1)\cup (M_3\setminus(M_1\cup M_2))\cup\cdots$.
		If $M$ is a compact WDC set in $\er^2$ then for every $x\in\partial M$ there is a $\rho_x>0$ as in Theorem~\ref{T:characterisationInPlaneNew}. 
		By the compactness of $\partial M$, we can find $x_1,\dots,x_n\in\partial M$ such that 
		\begin{equation*}
		\partial M=\bigcup_k \left(\partial M\cap\overline B\left(x_k,\frac{\rho_{x_k}}{2}\right)\right) \eqqcolon \bigcup_k  M_k.
		\end{equation*}
		Now, since $M$ is a disjoint union of $\partial M$ and $\INt M$, and $\INt M$ is a DC manifold of dimension $2$, it is enough to find a decomposition for $\partial M$ and therefore (by the argument above) it is enough to find it for each $M_k$.
		We will provide details only for the (most difficult) case of condition \ref{cond:manyConesNew} (from Theorem~\ref{T:characterisationInPlaneNew}).
		In that situation we have that 
		\begin{equation*}
		M_k=\bigcup_i \left(\partial S^k_i\cap\overline B\left(x_k,\frac{\rho_{x_k}}{2}\right)\right) \eqqcolon  \bigcup_i M^k_i,
		\end{equation*}
		where $S_i^k$ are the corresponding DC sectors from condition \ref{cond:manyConesNew}.
		Again, all sets $M^k_i$ are closed and so it suffices to find the decomposition for each of them separately.
		But this is easy since each $M_i^k$ is an isometric copy of a graph of a DCR function defined on a compact interval (cf. Remark~\ref{otsec}~(b2)).
		\item[(iii)] 
		It is also easy to see that, if a decomposition of type \eqref{decomp} exists, it is not uniquely determined. Moreover, an easy example shows (see \cite[Remark~7.11~(i)]{RZ}) that even in the case of a set of positive reach in $\R^2$, there is not always a  ``canonical decomposition''.
	\end{enumerate}
\end{remark}

\end{document}